\documentclass{amsart}
\usepackage[latin1]{inputenc}
\usepackage{amssymb}
\usepackage{amsmath}
\usepackage{enumerate}
\usepackage{epsfig,color}

\usepackage{color}

\usepackage{longtable,graphics,multirow,ulem}
\usepackage{tikz}
\usepackage{pgf,tikz}
\usetikzlibrary{arrows}
\usepackage[all]{xy}

\newcounter{nootje}
\newtheorem{theorem}{Theorem}[section]
\newtheorem{lemma}[theorem]{Lemma}
\newtheorem{proposition}[theorem]{Proposition}
\newtheorem{corollary}[theorem]{Corollary}

\newtheorem{definition}[theorem]{Definition}

\newtheorem{rem}[theorem]{Remark}

\numberwithin{equation}{section}
\newcommand{\rk}{\mbox{rank}}

\newcommand{\ra}{\rightarrow}

\newcommand{\Z}{\mathbb{Z}}

\newcommand{\Aut}{\mathrm{Aut}}

\makeatletter
\def\blfootnote{\xdef\@thefnmark{}\@footnotetext}

\author{Alice Garbagnati}
\address{Dipartamento di Matematica, Univ. Statale di Milano, Milan, Italy}
\email{alice.garbagnati@unimi.it}
\urladdr{ https://sites.google.com/site/alicegarbagnati/}
\author{Cec\'ilia Salgado}
\address{Instituto de Matem\'atica, Univ. Federal do Rio de Janeiro, Rio de Janeiro, Brazil}
\email{salgado@im.ufrj.br}
\urladdr{ http://www.im.ufrj.br/\~{}salgado}

\title[Linear systems on RES and elliptic fibrations on K3]{Linear systems on rational elliptic surfaces and elliptic fibrations on K3 surfaces}
\begin{document}

\keywords{Elliptic fibrations, Rational elliptic surfaces, K3 surfaces, Double covers.\\
The first author is partially supported by FIRB 2012 ``Moduli Spaces and their Applications".\\
The second author is partially supported by Cnpq- Brasil grant nb: 305743/2014-7 and the grant \textit{For Women in Science}, L' \'Oreal-Unesco-ABC}

\maketitle

\begin{abstract}
We consider K3 surfaces which are double cover of rational elliptic surfaces. The former are endowed with a natural elliptic fibration, which is induced by the latter. There are also other elliptic fibrations on such K3 surfaces, which are necessarily induced by special linear systems on the rational elliptic surfaces. We describe these linear systems. In particular, we observe that every conic bundle on the rational surface induces a genus 1 fibration on the K3 surface and we classify the singular fibers of the genus 1 fibration on the K3 surface it terms of singular fibers and special curves on the conic bundle on the rational surface.
\end{abstract}

\section{Introduction}

If one considers the classification of surfaces according to their Kodaira dimension, one finds examples of surfaces endowed with an elliptic fibration in the classes with $-\infty \leq \kappa\leq 1$. Nevertheless, thanks to their rich geometry, K3 surfaces are the only ones that might admit more than one elliptic fibration which is not of product type (\cite[Lemma 12.18]{SS}). For example, K3 surfaces obtained by a base change of a rational elliptic surface have in general more than one elliptic fibration on it. It is therefore natural to enquire about the classification of elliptic fibrations on elliptic K3 surfaces and, in case the K3 surface $X$ is obtained by a rational elliptic surface $R$, to investigate the relations among elliptic fibrations on $X$ and linear systems on $R$. 

In order to obtain the classification of elliptic fibrations on K3 surfaces, there is a lattice theoretic method (introduced by Nishiyama in \cite{Nis}) which allows one to describe the frame lattice of the fibration, i.e., the lattice spanned by the classes of the non-trivial components of the reducible fibers in the N\'eron-Severi lattice. 
The second, more geometric, method to classify elliptic fibrations on certain K3 surfaces was introduced by Oguiso in \cite{Oguiso} and later applied by \cite{Kloosterman} and \cite{CG}. It is strictly related with the presence of an involution $\iota$ on the K3 surface $X$ and so with the geometry of a different surface, $X/\iota$, which is the quotient of the K3 surface by this involution. This is our starting point, since we describe elliptic fibrations on the K3 surface $X$ which are induced by linear systems on the rational elliptic surface $X/\iota$. In the articles \cite{Oguiso}, \cite{Kloosterman} and \cite{CG} the involution $\iota$ acts trivially on the N\'eron--Severi group of $X$ and it fixes rational curves on $X$. Here we avoid both these hypothesis requiring only that the curves of highest genus fixed by $\iota$ has genus at most 1. This condition implies that $X$ is obtained as double cover of a rational elliptic surface $R$ (by a base change of order 2) and our purpose is to identify the linear systems on $R$ which induce elliptic fibrations on $X$.\\

Given 9 points in the intersection of two plane cubics, if the generic member of the pencil spanned by these is smooth, then the blow up $R$ of $\mathbb{P}^2$ in such points is naturally endowed with a (relatively minimal) elliptic fibration $\mathcal{E}_R:R\ra \mathbb{P}^1$. Moreover, this fibration admits a section, namely the last exceptional divisor of the series of blow-ups. An application of the canonical bundle formula assures that this is the unique elliptic fibration on $R$ (any fiber of an elliptic fibration in $R$ is algebraically equivalent to the anti-canonical divisor). On the other hand $R$ could admit several other fibrations $R\ra \mathbb{P}^1$, whose generic fibers are rational curves. 

Given the elliptic fibration described above $\mathcal{E}_R:R\rightarrow \mathbb{P}^1$ and two smooth fibers of this fibration, say $(F_R)_{b_1}$,$(F_R)_{b_2}$, the double cover of $R$ branched over $(F_R)_{b_1}$ and $(F_R)_{b_2}$ is a smooth surface, which is a K3 surface, naturally endowed with an elliptic fibration $\mathcal{E}_X:X\ra \mathbb{P}^1$, induced by $\mathcal{E}_R$. This construction can be generalized admitting the possibility that $(F_R)_{b_1}$ and $(F_R)_{b_2}$ are singular (possibly reducible) fibers without multiple components. In this case the double cover of $R$ branched on $(F_R)_{b_1}$ and $(F_R)_{b_2}$ is no longer smooth, but its minimal desingularization is again a K3 surface endowed with an elliptic fibration $\mathcal{E}_X$.

In contrast to what happens for $R$, $X$ admits more than one elliptic fibration, and only one of them is induced by $\mathcal{E}_R$. Since there is a $2:1$ map $X\ra R$, the rational fibrations on $R$, i.e. the conic bundles on $R$, induce fibrations on $X$. The Hurwitz formula applied to the restriction of the double cover map $X\ra R$ to the generic fiber of each conic bundle on $R$ assures that the fibrations in rational curves on $R$ induce fibrations in genus one curves on $X$, which are often endowed with a section. 
Hence a part of the elliptic fibrations on $X$ which are not induced by $\mathcal{E}_R$ are induced by conic bundles on $R$. 
Since all the conic bundles on $R$ induce elliptic fibrations on $X$, it is natural to relate the geometry of the conic bundles with the one of the induced elliptic fibrations. Theorems \ref{theo: reducible fibers induced by conic  bundle} and \ref{theo: extra singular fibers conic bundles} contain our main results on this  argument: for each elliptic fibration on $X$ induced by a conic bundle, the singular fibers are induced either by reducible fibers of the conic bundle or by smooth fibers of the conic bundle which have peculiar intersection properties with $(F_R)_{b_1}\cup (F_R)_{b_2}$ (i.e. they intersect at least one of the fibers $(F_R)_{b_1}$ $(F_R)_{b_2}$ in one point with multiplicity 2).
The singular fibers of the first type are classified in Theorem \ref{theo: reducible fibers induced by conic  bundle} and the singular fibers of the latter type are described in Theorem \ref{theo: extra singular fibers conic bundles}.

We saw above two constructions which produce elliptic fibrations on $X$: they can be induced by the elliptic fibration $\mathcal{E}_R$ or by conic bundles on $R$; but there might be also several other elliptic fibrations on $X$ which are neither induced by conic bundles nor by $\mathcal{E}_R$. We determine what are the conditions on the fibers of $R$ which assure that all the elliptic fibrations on $X$ are induced either by $\mathcal{E}_R$ or by conic bundles on $R$ (see Corollary \ref{corollary: Fb1Fb2 smooth+iota is on NS}) and we describe what are the other admissible linear systems on $R$ which can induce elliptic fibrations on $X$. Our main result in this sense is Theorem  \ref{theor: pencils induced on R by elliptic  fibrations on X} where we associate the elliptic fibrations on $X$ to three different types of linear systems on $R$: the (generalized) conic bundles, the splitting genus 1 pencils, and some special non-complete linear systems. The geometrical meaning of this distinction comes directly from the different actions of the cover involution $\iota$ on $X$. Indeed, by construction, $X$ is a double cover of $R$ and the cover involution $\iota$ on $X$ is non-symplectic. Given an elliptic fibration on $X$ there are three possibilities: $\iota$ acts as the elliptic involution on each fiber; $\iota$ acts as an involution on the basis; $\iota$ does not preserve the fibration. In the first case the elliptic fibration on $X$ is induced by a rational fibration on a certain blow up $\widetilde{R}$ of $R$ (i.e. by a generalized conic bundle). In the second case the elliptic fibration on $X$ is induced by a pencil of curves of genus 1 on $R$ which is not necessarily a fibration on $R$. We observe that if $X$ admits a fibration of this type, then $X$ can be realized also as double cover of another rational elliptic surface $R'$. In the third case, since the fibration is not preserved by $\iota$, the fibration is not associated to a complete linear system on $R$, but in any case one can describe some properties of the (non-complete) linear system obtained pushing down the elliptic fibration on $X$ to $R$. 

In the last part of the paper we analyze in detail some examples: in Section \ref{subsec: example K3 double cover R211} we present a well-known K3 surface for which we prove that all the elliptic fibrations are either induced by $\mathcal{E}_R$, or by conic bundles. In Section \ref{sec: example} we chose a rational elliptic surface $R$ and a K3 surface $X$, which is its double cover, and we describe on $X$ an elliptic fibration for each admissible type, indeed we describe an elliptic fibration on $X$ induced by: $\mathcal{E}_R$ (Section \ref{subsec: example R and X}); a conic bundle (Section \ref{subsec: example conic bunldes}); a generalized conic bundle (Section \ref{subsec: examples generalized  conic bundle});  a splitting genus 1 pencil (Section \ref{subsec: example splitting genus 1  pencil}); a non-complete linear system (Section \ref{subsec:  example type 3) fibration}).

{\it {\bf Acknowledgments}. 
The first author would like to thank Bert van Geemen, while the second thanks Jaap Top and Mathias Schuett for many interesting discussions. The authors would also like to acknowledge the organization of Women in Numbers Europe that was held in Luminy. The first discussions towards the ideas presented in this article started at that occasion.}

\section{Definitions and notation}

Let $k$ be an algebraically closed field of characteristic zero. 

\begin{definition} Let $X$ be a smooth projective algebraic surface. A surjective map $\mathcal{E}_X:X\ra C$ from $X$ to a smooth projective algebraic curve $C$ is called a genus 1 fibration if the generic fiber is a smooth curve of genus 1. A genus 1 fibration is called an elliptic fibration if it admits at least one section $s:C\ra  X$. 
An elliptic fibration is said to be relative minimal if there are no $(-1)$-curves which are components of fibers.
\end{definition}
In the following we always assume that the elliptic fibrations are relatively minimal, unless explicitly stated.

\begin{definition} A rational elliptic surface (RES) $R$ is a smooth rational surface endowed with an elliptic fibration. It is isomorphic to the blow up of $\mathbb{P}^2$ in 9 (possibly infinitely near)
points $P_1$, $P_2,\ldots,P_9$ such that the pencil of cubics
through these points contains at least one (and therefore infinitely many)
smooth cubic. The elliptic fibration on it is given by
$\mathcal{E}_R:R\ra \mathbb{P}^1$ whose fibers are the strict
transforms of the curves of the pencil of cubics of $\mathbb{P}^2$
having $P_1,\ldots,P_9$ as base points. \end{definition}

\begin{definition} A smooth projective algebraic surface $X$ is called a K3 surface if it has trivial canonical bundle and is regular, i.e., $h^{1,0}(X)=0$.\end{definition}
\begin{rem}{\rm If $\mathcal{E}:X\ra C$ is a genus 1 fibration on a K3 surface $X$, then $g(C)=0$, i.e. $C\simeq \mathbb{P}^1$, \cite[Theorem 6.12]{SS}.}\end{rem}

\subsection{K3 surfaces which are double covers of rational elliptic surfaces}

Let $f:\mathbb{P}^1\ra \mathbb{P}^1$ be a degree 2 map. It is branched in two points, say $b_1$ and $b_2$. Let us consider a rational elliptic surface $\mathcal{E}_R:R\ra \mathbb{P}^1$. We denote by $(F_R)_{b_1}$ and $(F_R)_{b_2}$ the fibers of $\mathcal{E}_R$ over $b_1$ and $b_2$.
We consider the fiber product $R\times_{\mathbb{P}^1}\mathbb{P}^1$ and the following diagram:

\begin{eqnarray}\xymatrix {X\ar[r]_{\beta}\ar[dr]_{\mathcal{E}_X}&R\times_{\mathbb{P}^1}\mathbb{P}^1\ar[d]\ar[r]&R\ar[d]_{\mathcal{E_R}}\\
&\mathbb{P}^1\ar[r]_{f}&\mathbb{P}^1}
\end{eqnarray}
The surface $R\times_{\mathbb{P}^1}\mathbb{P}^1$ naturally carries a map to $\mathbb{P}^1$ which is an elliptic fibration. Moreover, this surface is singular if and only if at least one of the fibers $(F_R)_{b_1}$ or $(F_R)_{b_2}$ is singular. In this case we denote by $\beta:X\ra R\times_{\mathbb{P}^1}\mathbb{P}^1$ its minimal resolution. This is a smooth surface with an elliptic fibration $\mathcal{E}_X:X\ra\mathbb{P}^1$ induced by $R\times_{\mathbb{P}^1}\mathbb{P}^1\ra\mathbb{P}^1$, and so by $\mathcal{E}_R:R\ra \mathbb{P}^1$.

By construction there is a generically $2:1$ map $X\ra R$ which is ramified on the strict transform (via $\beta$) of $(F_R)_{b_1}$ and $(F_R)_{b_2}$. We say that 
$(F_R)_{b_1}$ and $(F_R)_{b_2}$ are the branch fibers.

\begin{proposition}{\rm (\cite[Example 12.5]{SS})} The surface $X$ above is a K3 surface if and only if each of the fibers $(F_R)_{b_1}$ and $(F_R)_{b_2}$ is of type $I_n$ for $n\geq 0$ (where $I_0$ means that the fiber is smooth) or of type $II$, $III$, $IV$.\end{proposition}

From now on we assume that $\mathcal{E}_R:R\ra \mathbb{P}^1$ is a rational elliptic surface and that the elliptic surface $\mathcal{E}_X:X\ra \mathbb{P}^1$ obtained by the previous construction is a K3 surface.

\subsection{The branch fibers are stable}\label{subsec: R, tildeR and X, stable branch locus}

If the fiber $(F_R)_{b_1}$ (resp. $(F_R)_{b_2}$)
is smooth, i.e., a fiber of type $I_0$, then the fiber of
$R\times_{\mathbb{P}^1}\mathbb{P}^1\ra\mathbb{P}^1$ over
$f^{-1}(b_1)$ (resp. $f^{-1}(b_2)$) is smooth and so
the fiber of $\mathcal{E}_X:X\ra\mathbb{P}^1$ over $f^{-1}(b_1)$
(resp. $f^{-1}(b_2)$) is also smooth. 

If the fiber
$(F_R)_{b_1}$ (resp. $(F_R)_{b_2}$) is a fiber of type $I_n$,
$n>0$, then the fiber of
$R\times_{\mathbb{P}^1}\mathbb{P}^1\ra\mathbb{P}^1$ over
$f^{-1}(b_1)$ (resp. $f^{-1}(b_2)$) is a copy of a fiber of type
$I_n$ (i.e. it consists of $n$ rational curves meeting as a
polygon), but each singular point of $I_n$ is also a singular
point of the surface $R\times_{\mathbb{P}^1}\mathbb{P}^1$. These
singularities are of type $A_1$ and therefore it suffices to blow
them up once in order to construct $X$. Hence, the fiber of
$\mathcal{E}_X:X\ra\mathbb{P}^1$ over $f^{-1}(b_1)$ (resp.
$f^{-1}(b_2)$) is a fiber of type $I_{2n}$.

The discussion above gives us the following diagram:
\begin{eqnarray}\label{diagram fiber product}\xymatrix{X\ar[d]\ar[r]_{2:1}^{\pi}\ar@{-->}[ddr]^{2:1}&\widetilde{R}\ar[d]^{\beta_2}\\
R\times_{\mathbb{P}^1}\mathbb{P}^1\ar[r]&R\ar[d]^{\beta_1}\\
&\mathbb{P}^2}\end{eqnarray}
where $\beta_2:\widetilde{R}\ra R$ is the blow up of $R$ in all the singular points of the fibers $(F_R)_{b_1}\subset R$ and $(F_R)_{b_2}\subset R$ and $\beta_1:R\ra \mathbb{P}^2$ is the blow up of $\mathbb{P}^2$ in the points $P_1,\ldots, P_9$. Let us assume that the fiber $(F_R)_{b_1}$ is of type $I_{n_{b_1}}$ and the fiber $(F_R)_{b_2}$ is of type $I_{n_{b_2}}$ where $n_{b_1}$ and $n_{b_2}$ are non-negative integers. Then $\widetilde{R}$ is isomorphic to a blow up of $\mathbb{P}^2$ in $9+n_{b_1}+n_{b_2}$ (possibly infinitely near) points $P_1, P_2\ldots P_9,Q_1,Q_2\ldots Q_{n_{b_1}+n_{b_2}}$. So $X$ is a double cover of a blow up of $\mathbb{P}^2$ branched along the strict transform of the sextic which is the union of the two (possibly reducible) cubics $\beta_1((F_R)_{b_1})$ and $\beta_1((F_R)_{b_2})$.

The map $\mathcal{E}_R\circ\beta_2:\widetilde{R}\ra \mathbb{P}^1$ is an elliptic fibration which is, in general, not relatively minimal. This is the case if and only if $\beta_2$ is the identity.

\subsection{The branch fibers can be unstable}

If at least one of the branch fibers is unstable (i.e. of type $II$, $III$ or $IV$) we still construct $X$ as double cover of a blow up $\widetilde{R}$ of $R$, where $\widetilde{R}$ is a blow up of the singular points of the fiber $II$, $III$, $IV$ and thus also in this case we obtain that $\widetilde{R}$ is isomorphic to a blow up of $\mathbb{P}^2$ in $9+r$ (possibly infinitely near) points $P_1, P_2\ldots P_9,Q_1,Q_2,\ldots, Q_{r}$ for a certain value of $r$ (which depends on the type of the fibers $(F_R)_{b_1}$, $(F_R)_{b_2}$).  If $(F_R)_{b_i}$ is a fiber of type $IV$, then $\beta_2:\widetilde{R}\ra R$ introduces 4 exceptional curves over the unique singular point of  $(F_R)_{b_i}$, hence the construction of $\widetilde{R}$ is slightly different with respect to the stable case, but in both the situation $\widetilde{R}$ is the minimal blow up of $R$ such that the branch locus of $X\ra\widetilde{R}$ is smooth.

We present a table with the singular fibers on the induced elliptic fibration on $X$ on the next paragraph.

\subsection{The non-symplectic involution $\iota$}
By construction $X$ is a double cover of the smooth (possibly non-minimal) surface $\widetilde{R}$.  The surface $X$ is therefore endowed with a cover involution. We denote it by $\iota$ and observe that $X/\iota\simeq\widetilde{R}$. The involution $\iota$ is non-symplectic, i.e. it does not preserve the symplectic structure of $X$ and thus it is not trivial
on $H^{2,0}(X)$. Its fixed locus is
the ramification locus of the smooth double cover $X\ra\widetilde{R}$.
In particular, we have the following table.

 \begin{eqnarray*}\begin{array}{|c|c|c|} \hline \text{ Fiber on branch locus on } R &\text{Fiber induced on } X&
\# \text{ components fixed by } \iota \\
\hline
I_0 & I_0 & 1 \text{ genus } 1\, \text{ curve} \\
\hline
I_n & I_{2n} & n  \text{ rational curves }\\
\hline
II & IV & 1 \text{ rational curve }\\
\hline
III & I_0^* & 2  \text{ rational curves }\\
\hline 
IV & IV^* & 4  \text{ rational curves }\\
\hline
\end{array}\end{eqnarray*}

The reader can consult \cite[VI.4.1]{Mi} for more details on the fiber types after base change.

\section{Bundles over $\widetilde{R}$ and induced genus 1-fibration on $X$}

First we recall the definition of conic-bundles on a
rational surface. Then we focus on rational elliptic surfaces and generalize the concept to some similar
objects which are useful for our purpose. 

\begin{definition}
A conic bundle on the surface $R$ is a surjective morphism $g: R \rightarrow \mathbb{P}^1$ such that the generic fiber is a smooth rational curve.
\end{definition}

We give an interpretation of the definition above in terms of classes of divisors on $\mathrm{NS}(R)$. In the following we denote by $(F_R)_t$ the fiber of the elliptic fibration $\mathcal{E}_R:R\ra\mathbb{P}^1$ over $t\in\mathbb{P}^1$.

Let $D$ be a smooth fiber of a conic bundle on $R$. Since the class of $D$ gives the class of a fiber, it is \textit{nef} and has trivial self-intersection. Moreover, since it is rational, adjunction implies that $D\cdot K_R=-2$ and since $F_R=-K_R$, one obtains $D \cdot F_R=2$.

On the other hand, given a \textit{nef} class $D$ with the above intersection properties, the induced map $|D|: R\rightarrow \mathbb{P}^1$ is a conic bundle.

From the above, on a surface endowed with an elliptic fibration, we have the following equivalent definition of conic bundle.
\begin{definition}\label{defi: conic bundle}
A conic bundle on $R$ is a \textit{nef} class $D$ in $\mathrm{NS}(R)$ such that 
\begin{itemize}
\item[i)] $D\cdot(F_R)_t= D\cdot(-K_R)= 2$.
\item[ii)] $D^2=0$.
\end{itemize}
\end{definition}

Denote by $D_s$ the fibers of the conic bundle $|D|$. These fibers are mapped by $\beta_1$ to a pencil of rational curves passing with a certain multiplicity through the points $P_1,\ldots ,P_9$. Moreover, since $D_s$, $s\in\mathbb{P}^1$, is a bisection of $\mathcal{E}_R$, it meets both $\beta_1((F_R)_{b_1})$ and $\beta_1((F_R)_{b_2})$ in two points each.\\

Since we have a $2:1$ map from $X$ to $\widetilde{R}$, which is in general a blow up of $R$, we need a generalization of the previous definition which extends the notion of conic bundles to $\widetilde{R}$.

\begin{definition}\label{defi: generalized conic bundle}
A generalized conic bundle on $\widetilde{R}$ is a \textit{nef} class $D$ in $\mathrm{NS}(\widetilde{R})$ such that 
\begin{itemize}
\item[i)]  $D\cdot(-K_{\widetilde{R}})= 2$.
\item[ii)] $D^2=0$.
\end{itemize}
\end{definition}

Note that $\tilde{R}$ is endowed with a not relatively minimal elliptic fibration induced by the elliptic fibration on $R$. We call $F_{\widetilde{R}}$ the class of a smooth fiber. The difference between the Definitions \ref{defi: conic bundle} and \ref{defi: generalized conic bundle} is that since $-K_{\widetilde{R}}$ is not the class of a fiber, $D \cdot (-K_{\widetilde{R}})$ is not necessarily equal to $D\cdot F_{\tilde{R}}$.

A conic bundle on $R$ induces a generalized conic bundle on
$\widetilde{R}$. If $R \neq \widetilde{R}$ then there are generalized conic bundles on $\widetilde{R}$ which do not induce conic bundle on $R$.

A generalized conic bundle is mapped by $\beta_1\circ\beta_2$ to a pencil of rational curves passing with certain multiplicities through the points $P_1,\ldots ,P_9,Q_1\ldots Q_r\subset\mathbb{P}^2$.

\begin{rem}{\rm The notion of conic bundle on $R$ is clearly independent on $f$ and on the branching points $b_1$ and $b_2$. On the other hand the notion of generalized conic bundle strictly depends on $b_1$ and $b_2$ and in particular on the properties of the fibers $(F_R)_{b_1}$ and $(F_R)_{b_2}$.}\end{rem}

For similar reasons as above, we also need a generalization of the notion of genus 1 pencils.

\begin{definition}
A splitting genus 1 pencil on $\widetilde{R}$ is a proper morphism $\varphi: \widetilde{R}\rightarrow \mathbb{P}^1$ such that
\begin{itemize}
\item[i)] $C_s=\varphi^{-1}(s)$ is a smooth genus 1 curve for almost all $s \in \mathbb{P}^1$.
\item[ii)] $(C_s)(-K_{\widetilde{R}})=0$ for almost all $s\in \mathbb{P}^1$.
\end{itemize}

We also call the pencil of curves $\{C_s\}_{s\in \mathbb{P}^1}$ as above, a splitting genus 1 pencil on $\widetilde{R}$.
\end{definition}
Since $C_s$ is generically a smooth curve of genus 1 and $C_sK_{\widetilde{R}}=0$, by adjunction one has $C_s\cdot C_s=0$. 
A priori, given a smooth genus 1 curve $C_s$ such that $C_s\cdot K_{\widetilde{R}}=0$, the Riemann--Roch Theorem assures only that $h^0(\widetilde{R},\mathcal{O}_{\widetilde{R}}(C_s))\geq 1$, so nothing guarantees that there exists a 1-dimensional linear system of curves $C_s$. On the other hand we know that the previous definition is not empty, indeed, denoting by $F_R$ a generic smooth fiber of the fibration $\mathcal{E}_{R}:R\ra\mathbb{P}^1$, $\beta_2^*(F_R)$ is a smooth curve of genus 1, which moves in a 1-dimensional linear system, i.e. $\varphi_{|\beta_2^*(F_R)|}:\widetilde{R}\ra\mathbb{P}^1$ is a (not necessarily relatively minimal) elliptic fibration and satisfy the definition of splitting genus 1 pencils. Hence on $\widetilde{R}$ there is at least one splitting genus 1 pencil, namely the one induced by the elliptic fibration $\mathcal{E}_R:R\ra\mathbb{P}^1$ by pull back.

The reason of the word ``splitting" in the definition is explained by the following lemma.
\begin{lemma}\label{lemma on splitting genus 1 pencil}
Let $\varphi: \widetilde{R}\rightarrow \mathbb{P}^1$ be a genus 1 pencil on $\widetilde{R}$ and $C_s$ be a generic fiber of $\varphi$. Then $\pi^{-1}(C_s)$ splits in the union of two curves, both isomorphic to $C_s$.
\end{lemma}
\proof
The generic fiber $C_s$ is a smooth curve of genus 1 on $\widetilde{R}$ and thus, by adjunction, its class has trivial self-intersection. Therefore $\pi^*(C_s)$ has also trivial self-intersection on $X$. Moreover $\pi^{-1}(C_s)$ is either a connected smooth curve or the union of two curves, both isomorphic to $C_s$. In the first case $\pi^{-1}(C_s)$ is a smooth genus 1 curve, by adjunction, since its class has trivial self-intersection. In this case $\pi$ restricted to $\pi^{-1}(C_s)$ is a 2:1 map between smooth genus 1 curves, and so it is unramified. This means that $C_s$ does not meet the branch locus of $\pi$ and that $\varphi_{|\pi^{-1}(C_s)|}:X\ra\mathbb{P}^1$ is a genus 1 fibration. The cover involution $\iota\in \Aut(X)$ preserves the fibers of the fibration $\varphi_{|\pi^{-1}(C_s)|}:X\ra\mathbb{P}^1$ and acts on the generic fiber as a fixed point free involution. A fixed point free involution on an elliptic curve preserves the period of the curve, and thus $\iota$ preserves the period of $X$. As a consequence $\iota$ is a symplectic involution on the K3 surface $X$. But this is impossible, because $X/\iota\simeq \widetilde{R}$ and $\widetilde{R}$ is a rational surface. In particular $\widetilde{R}$ does not have a symplectic structure and so $\iota$ does not preserves the symplectic structure of $X$. We conclude that for each splitting genus 1 fibration, $\pi^{-1}(C_s)$ splits in the union of two curves, both isomorphic to $C_s$.
\endproof

By construction, a splitting genus 1 pencil on $\widetilde{R}$ is mapped by $\beta_1\circ\beta_2$ to a pencil of (not necessarily smooth) curves of genus 1 on $\mathbb{P}^2$  passing through the points $\{P_1,\ldots, P_9,Q_1,\ldots, Q_r\}$ with a certain multiplicity. All the curves of the pencil of plane curves $(\beta_1\circ\beta_2)(C_s)$ intersect $\beta_1((F_R)_{b_1}\cup (F_R)_{b_2})$ with even multiplicity at each intersection point not contained in the set $\{P_1,\ldots, P_9,Q_1,\ldots, Q_r\}$, since they split in the double cover. This imposes strong conditions on $\{C_s\}$ and $(F_R)_{b_1}\cup (F_R)_{b_2}$.

Viceversa under certain conditions the choice of a pencil of, not necessarily smooth, genus 1 plane curves induces a splitting genus 1 pencil on $\widetilde{R}$; for example one can take a pencil of smooth cubics in $\mathbb{P}^2$ as follows:
The surface $\widetilde{R}$ is isomorphic to the blow up of $\mathbb{P}^2$ in the points $P_1,\ldots, P_9,Q_1, \ldots, Q_r$, and $r\geq 2$. A choice of $9$ points in $\{P_1,\ldots, P_9,Q_1,\ldots,
Q_r\}$ general enough is associated to a splitting
genus 1 pencil on $\widetilde{R}$. To be more precise, let us
choose nine points among $\{P_1,\ldots, P_9,Q_1,\ldots,
Q_r\}$ such that there is a pencil of cubics
through these nine points. If this pencil contains at least one
(and so infinitely many) smooth curves, then the choice is general
enough and gives a splitting genus 1
pencil on $\widetilde{R}$. Indeed, if the curves of the pencil pass simply through nine points in $\{P_1,\ldots, P_9,Q_1, \ldots,
Q_r\}$, then the intersection between the class of one of these curves and the canonical bundle of $\widetilde{R}$ is automatically zero. In particular every choice of nine
points in $\{P_1,\ldots, P_9,Q_1, \ldots, Q_r\}$ which
is general enough and contains at least two points $Q_i, Q_j$ is
associated to a splitting genus 1 pencil which is not the pencil
of cubics associated to $\mathcal{E}_R:R\ra \mathbb{P}^1$. We see below why having only one extra point $Q_1$ is not enough to produce different splitting genus 1 pencils.

\begin{proposition}\label{prop: conditions to exclude splitting genus 1 smooth cubics}
Let $R$ be a rational elliptic surface isomorphic to the blow up of $\mathbb{P}^2$ in $P_1,\ldots, P_9$. Let $\widetilde{R}$ be the blow up of $R$ at some singular points of fibers of $R$. Let $Q_1,\ldots, Q_r$ be the image of these singular points in $\mathbb{P}^2$ by the composition of the two blow up maps above. If one of the following holds
\begin{itemize}
\item[i)] $r=1$, or
\item[ii)] the points $Q_1,\ldots, Q_r$ lie all in the same singular cubic below a fiber of $R$
\end{itemize}
then there is only one splitting genus 1 pencil on $\widetilde{R}$ which is induced by a pencil of smooth cubics in $\mathbb{P}^2$, namely the one induced by the elliptic fibration on $R$.  
\end{proposition}
\begin{proof}
The points $P_1,\ldots, P_9$ are the base points of a pencil of cubics which gives the elliptic fibration on $R$. If $r=1$, then a splitting genus 1 pencil on $\widetilde{R}$ is the pull back of a pencil of cubics in $\mathbb{P}^2$ through nine points among $P_1,\ldots, P_9, Q_1$. To conclude in case $i)$, it suffices to observe that any cubic that passes through eight points among $P_1,\ldots, P_9$ also passes through the ninth.
For $ii)$, note that the points $Q_i$'s lie below singularities of a fiber and are therefore singular points on a cubic that passes through $P_1,\ldots, P_9$. If $C$ is a smooth cubic through $9=s+t$ points, where $s$ points are chosen among $P_1, \ldots, P_9$ and $t$ points are among $Q_1,\ldots, Q_r$, then $C$ intersects the singular cubic above with multiplicity at least $s+ 2t>9$, which is absurd if $t>0$. Hence the only splitting genus 1 fibration induced by cubics is the one given by cubics through $P_1,\ldots, P_9$, i.e., the one that is already an elliptic fibration on $R$.
\end{proof}

Pencils of smooth cubics are not the only ones responsible for the existence of splitting genus 1 fibrations. Indeed there might exist splitting genus 1 pencils induced by a system of non smooth genus 1 curves on $\mathbb{P}^2$ (see Section \ref{subsec: example splitting genus 1 pencil}).

\begin{proposition}\label{prop: (generalized) conic bundles and splitting genus 1 pencil induces genus 1 fibration} Each conic bundle, generalized conic bundle and splitting genus 1 pencil on $\widetilde{R}$ induces a genus 1-fibration on
$X$.\end{proposition} \proof Since every conic bundle on $R$
induces a generalized conic bundle on $\widetilde{R}$, it suffices
to prove the result for generalized conic bundles in order to
obtain the proof also for conic bundles.

Let  $\{C_s\}_{s\in\mathbb{P}^1}$ be a generalized conic bundle on
$\widetilde{R}$. The curve $C_s$ is a smooth rational curve for almost every
$s\in\mathbb{P}^1$ and $C_s(-K_{\widetilde{R}})=2$. By adjunction,
this implies that $[C_S]^2=0$ in $NS(\widetilde{R})$. Hence the
class of the curves $\pi^{-1}(C_s)$ in $NS(X)$ has trivial
self-intersection. Since this is clearly an effective class, we have a genus 1 fibration
$\varphi_{|\pi^{-1}(C_s)|}:X\ra\mathbb{P}^1$.

Let $\{C_s\}_{s\in\mathbb{P}^1}$ be a splitting genus 1 pencil of curves on $\widetilde{R}$. By Lemma \ref{lemma on splitting genus 1 pencil}, for almost every $s\in\mathbb{P}^1$, the curves $C_s$
are smooth and $\pi^{-1}(C_s)$ is given by two disjoint copies of $C_s$ and $\varphi_{|\pi^{-1}(C_s)|}:X\ra\mathbb{P}^1$ is a genus 1 fibration.
\endproof

\section{Elliptic fibrations on $X$ and induced system of curves on $\widetilde{R}$ and on $R$}

In this section we first study what are the linear systems of curves induced on $\widetilde{R}$ by the elliptic fibrations on $X$. We classify three possible types of elliptic fibrations on $X$ according to the action of $\iota$. We then associate each of these types to a linear system on $\widetilde{R}$, induced by the elliptic fibration. This first goal of this section is summarized in Theorem \ref{theor: pencils induced on R by elliptic  fibrations on X}. In the second part of this section, we study what are the properties of $X$, $\iota$ and $\mathcal{E}_R$ which allow one to conclude that $X$ admits or not certain types of elliptic fibrations. The corollaries \ref{corollary: iota is id on NS}, \ref{corollary: Fb1 and Fb2 smooth}, \ref{corollary: Fb1Fb2 smooth+iota is on NS} give a first step towards a classification of the elliptic fibrations on K3 surfaces which are double cover of rational elliptic surfaces. In particular, in Proposition \ref{prop NS in case iota is id on NS}, we assume certain properties on the pair $(X,\iota)$ and we identify the N\'eron--Severi group of $X$ and the properties of a rational elliptic surface $R$ of which $X$ is the cover.

\subsection{Elliptic fibrations of $X$ and action of $\iota$}
Let $\mathcal{E}^{(n)}:X\ra\mathbb{P}^1$ be an elliptic fibration
on $X$, not necessarily equal to $\mathcal{E}_X$ (the elliptic
fibration induced on $X$ by $\mathcal{E}_R$). We denote by
$[F^{(n)}]$ the class of the fiber of this fibration and by
$F^{(n)}_t$ the fiber of this fibration over the point
$t\in\mathbb{P}^1$.
We recall that $X$ is naturally endowed with a non-symplectic involution $\iota$, which is the cover involution of the double cover $\pi:X\ra \widetilde{R}$.
The action of $\iota$ on an elliptic fibration $\mathcal{E}^{(n)}$ can be of three different types: \begin{enumerate}
\item $\iota$ preserves the fibers
of the fibration $\mathcal{E}^{(n)}$, i.e.
$\iota(F^{(n)}_t)=F^{(n)}_t$ for every $t\in\mathbb{P}^1$. In this
case the action of $\iota$ on the basis of the fibration is
trivial.
\item $\iota$ does not preserve the
fibers of the fibration, but $\iota^*$ preserves the fibration and
in particular the class of a fiber, i.e.
$\iota(F^{(n)}_t)=F^{(n)}_{t'}$ for certain values $t\neq t'$ but
$\iota^*([F^{(n)}])=[F^{(n)}]$. In this case $\iota$ restricts to
an involution of the basis of the fibration. We say that $\iota$
preserves the fibration. 
\item $\iota$ does not preserve the
fibration, and in particular $\iota^*([F^{(n)}])=[F^{(n')}]$ where
$[F^{(n')}]$ is the class of another elliptic fibration on $X$,
$\mathcal{E}^{(n')}:X\ra\mathbb{P}^1$, which is not
$\mathcal{E}^{(n)}$. In particular $[F^{(n)}]\neq [F^{(n')}]\in
NS(X)$.\end{enumerate}

\begin{definition} We say that an elliptic fibration on $X$ is of type 1), 2), 3) with respect
to $\iota$ if $\iota$ preserves the fibers, preserves the fibration
but not the fibers, and does not preserve the fibration
respectively.
\end{definition}

\begin{theorem}\label{theor: pencils induced on R by elliptic fibrations on X}
Let $\mathcal{E}^{(n)}:X\ra\mathbb{P}^1$ be an elliptic fibration on $X$.
If $\mathcal{E}^{(n)}$ is of type 1) with respect to $\iota$ then
$\mathcal{E}^{(n)}$ is induced by a generalized conic bundle on
$\widetilde{R}$.

If $\mathcal{E}^{(n)}$ is of type 2) with respect to $\iota$, then
$\mathcal{E}^{(n)}$ is induced by a splitting genus 1 pencil.

If $\mathcal{E}^{(n)}$ is of type 3) with respect to $\iota$, then
$\mathcal{E}^{(n)}$ does not induce a fibration on $\widetilde{R}$
and, more precisely, the image $\pi(F^{(n)}_t)$, $t\in\mathbb{P}^1$
is a 1-dimensional non-complete linear system.

\end{theorem}

\proof Let us consider the generic (smooth) fiber $F^{(n)}_t$ of
$\mathcal{E}^{(n)}$.

If $\mathcal{E}^{(n)}$ is of type 1) with respect to $\iota$, then
$\iota$ preserves the curve $F^{(n)}_t$ and acts on it by fixing 4
points. Indeed the involutions on a smooth curve of genus 1 are either
fixed points free or fix 4 points. The ones which are fixed
point free are translations and preserve the period of the
elliptic curve. But since $\iota$ acts as the identity on the base
of the fibration $\mathcal{E}^{(n)}$, if $\iota$ restricted to the
fibers preserves the period of the fibers, then $\iota$ preserves
the period of the surface $X$. This is not possible, because
$\iota$ is a non-symplectic involution. So $\iota$ acts on each
fiber with 4 fixed points (i.e. it acts as the hyperelliptic
involution, possibly composed with some translations). Let us
now consider the image of $F^{(n)}_t$ on $\widetilde{R}$ under the quotient
map $\pi:X\ra \widetilde{R}\simeq X/\iota$. The curve
$D_t^{(n)}:=\pi(F^{(n)}_t)$ is a rational curve, since
$\pi:F^{(n)}_t\ra D^{(n)}_t$ is a $2:1$ cover branched in 4
points. So $\mathcal{E}^{(n)}$ induces a pencil of rational curves
on $\widetilde{R}$ and we denote by $[D^{(n)}]$ the class of the
fiber of this pencil. The self-intersection $[D^{(n)}]^2=0$, since $[D^{(n)}]$ is the class of a fiber of a fibration. On the other hand $g(D^{(n)})=0$ and so, by adjunction, $D^{(n)}K_{\widetilde{R}}=-2$. Thus the pencil $D^{(n)}_t$ is a generalized conic bundle.

If $\mathcal{E}^{(n)}$ is of type 2) with respect to $\iota$, then $\iota$ is an involution of the basis. Since the basis of the fibration is $
\mathbb{P}^1$, the involution $\iota$ fixes two points on it, say $t_1$ and $t_2$. It preserves therefore two fibers of the fibration and switches the others. Let $F^{(n)}_t$ and $F^{(n)}_{t'}$ be two smooth fibers such that $\iota(F^{(n)}_t)=F^{(n)}_{t'}$ and consider $C^{(n)}_t:=\pi(F^{(n)}_t)$. The inverse image of $C^{(n)}_t$ by $\pi$ consists of two disjoint smooth genus 1 curves (the fibers $F^{(n)}_t$ and $F^{(n)}_{t'}$). Hence $C^{(n)}_t$ is a copy of $F^{(n)}_t$ and is therefore a smooth curve of genus 1. Moreover $[C^{(n)}_t][C^{(n)}_t]=0$, since $[F^{(n)}_t][F^{(n)}_t]=0$. So $\mathcal{E}^{(n)}$ induces a genus 1 fibration on $\widetilde{R}$. By adjunction, since $C^{(n)}_t$ is a smooth curve of genus 1 and $[C^{(n)}_t][C^{(n)}_t]=0$, we have $[C^{(n)}_t]K_{\widetilde{R}}=0$, and therefore $\mathcal{E}^{(n)}$ induces a genus 1 pencil on $\widetilde{R}$. 

If $\mathcal{E}^{(n)}$ is of the type 3) with respect to $\iota$,
then there exists another elliptic fibration $\mathcal{E}^{(n')}$ on
$X$ such that $\iota^*([F^{(n)}])=[F^{(n')}]$. Since
$[F^{(n)}]\not \sim [F^{(n')}]$, where $\sim$ is the linearly
equivalence, and both $[F^{(n)}]$ and $[F^{(n')}]$ are classes of
effective divisors supported on smooth curves, it holds that $[F^{(n)}][F^{(n')}]>0$. Let us
denote by $m$ the positive integer $m:=[F^{(n)}][F^{(n')}]$. We
observe that $[F^{(n)}+F^{(n')}]^2=2m>0$ and that
$[F^{(n)}+F^{(n')}]$ is an effective divisor. The linear system
$|F^{(n)}|$ has no base points, hence $|F^{(n)}+F^{(n')}|$ is a
complete linear system base points free. In particular there
exists a smooth genus $m+1$ curve $C_X$ whose class is
$[F^{(n)}+F^{(n')}]$. Therefore
$|C_X|=|F^{(n)}+F^{(n')}|$ is a $m+1$-dimensional complete linear
system whose general element is smooth (cf.\ \cite[Proposition
2.6]{SD}). On the other hand, if we consider the set of curves of
the form $F^{(n)}_t+F^{(n')}_t$, this is a 1-dimensional space,
parametrized by $\mathbb{P}^1_t$, whose generic element is a
singular reducible curve (indeed it splits in the sum of two
smooth genus 1 curve meeting in $m$ points). Let us denote by
$D_t:=\pi(F^{(n)}_t)=\pi(F^{(n')}_t)$. 
So $F^{(n)}_t\ra\mathbb{P}^1$ induces a pencil of curves
$D_t\ra\mathbb{P}^1$ on $\widetilde{R}$. This pencil is a non-complete sub-linear system of the complete linear system induced
by the $(m+1)$-dimensional linear system $|C_X|$ on $X$. Indeed $\pi(C_X)$ is a smooth curve on $\widetilde{R}$ and $[\pi(C_X)]=[\pi(F^{(n)})_t+\pi(F^{(n')})_t]$ so that its linear system contains the non-complete sub-linear system given by $2D_t$.\endproof

Let $\mathcal{E}^{(n)}$ be an elliptic fibration on $X$ of type 2). 
Then $\iota$ generically switches pairs of fibers and preserves exactly two fibers, denoted by $F_{t_1}^{(n)}$ and $F_{t_2}^{(n)}$.
The fibration $\mathcal{E}^{(n)}$ induces a non necessarily relatively minimal elliptic fibration on $\widetilde{R}$ such that
each fiber of this elliptic fibration, with the exception of $\pi(F^{(n)}_{t_1})$ and $\pi(F^{(n)}_{t_2})$, splits in the double cover $\pi:X\ra \widetilde{R}$. So each curve $C^{(n)}_t:=\pi(F_t^{(n)})\subset \widetilde{R}$, possibly with the exception of $t=t_1$ and $t=t_2$, intersects the branch locus of $\pi:X\ra \widetilde{R}$ with even multiplicity in each intersection point, thus $C^{(n)}_t$ intersects $\widetilde{(F_R)_{b_1}}\cup \widetilde{(F_R)_{b_2}}$ with even multiplicity in each intersection point. Now we consider the blow down $\beta_1\circ\beta_2:\widetilde{R}\ra\mathbb{P}^2$. It sends $\widetilde{(F_R)_{b_i}}$, $i=1,2$ to $\beta_1({(F_R)_{b_i}})$, $i=1,2$ and $C^{(n)}_t$ to a not necessarily smooth curve of genus 1 of $\mathbb{P}^2$ passing through at least 9 points among $\{P_1,\ldots, P_9,Q_1,\ldots Q_r\}$. Since $\beta_1\circ\beta_2\circ\pi:X\ra\mathbb{P}^2$ is a 2:1 cover branched along  $\beta_1((F_R)_{b_1})\cup \beta_1((F_R)_{b_2})$, and the generic fiber of the pencil of cubics  $\{C^{(n)}_t\}_t\in\mathbb{P}^1$ splits on $X$, we deduce that almost all the fibers of $\{C^{(n)}_t\}_t\in\mathbb{P}^1$ intersect $\beta_1((F_R)_{b_1})\cup \beta_1((F_R)_{b_2})$ with an even multiplicity in each intersection point. This characterizes the pencil of curves on $\mathbb{P}^2$ which are induced by genus 1 fibration of second type on $X$. Viceversa, we already observed that a pencil of not necessarily smooth curves of genus 1 on $\mathbb{P}^2$ induces, under some conditions, a splitting genus 1 pencil on $\widetilde{R}$ and thus a genus 1 fibration of second type on $X$.

Let $\mathcal{E}^{(n)}$ be an elliptic fibration of type 3), $F^{(n)}_t$ be a generic fiber of $\mathcal{E}^{(n)}$ and $D_t:=\pi(F^{(n)}_t)$. Denote by $F^{(n')}_t$ the curve $\iota(F^{(n)}_t)$ and $\mathcal{E}^{(n')}$ the elliptic fibration whose fiber over $t$ is $F^{(n')}_t$. Then $D_t=\pi(F^{(n)}_t)=\pi(F^{(n')}_t)$.
By the projection formula, $\left(\pi^*(D_t)\right)^2=2 \left(D_t\right)^2$, so $(F^{(n)}_t+F^{(n')}_t)^2=2(D_t)^2$ and thus $(D_t)^2=m$. Let us denote by $B_R$ the branch locus of $\pi:X\ra\widetilde{R}$ and $B_X=\pi^{-1}(B_R)$. Then $(F^{(n)}_t)\cdot B_X=(F^{(n')}_t)\cdot B_X=m$ and $D_t\cdot B_R=m$. Since $X$ has trivial canonical bundle, the canonical bundle of $\widetilde{R}$ is $K_{\widetilde{R}}=-B_R/2$ so that $D_t\cdot (-K_{\widetilde{R}})=D_t\cdot (2B_R)=2m$.

Let us now consider $C_X$, a smooth curve in the complete linear system $|F^{(n)}+F^{(n')}|$. It is smooth of genus $m+1$ and its image under $\pi$ is the smooth curve $C_R:=\pi(C_X)$. The map $\pi$ restricted to $C_X$ is a 2:1 cover $C_X\ra C_R$, branched in $C_X\cdot B_X=2m$ points. We are able to compute the genus of $C_R$ by the Riemann--Hurwitz formula: 
\begin{eqnarray}\label{eq: Riemann-Hurwitz}2g(C_X)-2=2(2g(C_R)-2)+2m,\end{eqnarray}
which implies that $g(C_R)=1$.

Similar results can be obtained also by adjunction formula. Indeed $(C_R)^2=m$, since $2m=(C_X)^2=\left(\pi^*(C_R)\right)^2=2 \left(C_R\right)^2$ and $C_R\cdot K_R=-(C_R\cdot B_R)/2=-m$.

Thus the singular curves $D_t\subset\widetilde{R}$ are a sub-system of a complete linear system of generically smooth curves of genus 1, meeting the branch locus in $2m$ points. 

We observe that the non-complete linear system associated to fibration of type $3)$ are a generalization of the complete linear system associated to the fibration of type 2), i.e. of the splitting genus 1 pencil. Indeed, if we consider the description of the linear system associated to fibration of type 3) and we allow $m$ to be 0, then the curves $D_t$ defined above are such that $D_t$ is a curve of genus 1 and $D_t\cdot K_{\widetilde{R}}=0$, thus $D_t^2=0$. This is exactly what we require to define the splitting genus 1 pencil. This is consistent with the geometric interpretation of these linear systems. Indeed if we allows $m$ to be 0, this means that $F^{(n)}\cdot \iota(F^{(n)})=0$. The consequence is that $F^{(n)}$ and $\iota(F^{(n)})$ are classes of the fiber of the same fibration, and so $\iota$ preserves the fibration defined by $F^{(n)}$. By definition this means that the fibration $\mathcal{E}^{(n)}$ is of type 2) and hence associated to a splitting genus 1 pencil. 

\subsection{Admissible elliptic fibrations on $X$}
We now analyze deeply the geometric properties related to the presence (or the absence) of a certain type of elliptic fibration on $X$. In particular, we show that elliptic fibrations of type 3) cannot appear if $\iota$ is the identity on the N\'eron--Severi group of $X$, which implies in turn that $R$ has at most two reducible fibers, and that the fibrations of type 2) are reduced to the fibration $\mathcal{E}_X$ if the cover $X\dashrightarrow R$ is branched on at least one smooth fiber of the fibration $\mathcal{E}_R$.

\begin{corollary}\label{corollary: iota is id on NS}
If the map $\iota$ restricted to the N\'eron--Severi group of
$X$ is the identity, then there are no reducible fibers of
$\mathcal{E}_R\ra \mathbb{P}^1$ outside possibly $(F_R)_{b_1}$
and $(F_R)_{b_2}$. 

If $\iota$ is the identity on $NS(X)$, then all elliptic
fibrations on $X$ are induced either by generalized conic bundles
on $\widetilde{R}$ or by splitting genus 1 pencils.
\end{corollary}
\proof If $\mathcal{E}_R$ has at least one reducible fiber outside
$(F_R)_{b_1}$ and $(F_R)_{b_2}$, then the elliptic fibration
$\mathcal{E}_X$ has two fibers of the same type, switched by
$\iota$. Since the classes of irreducible components of two
different reducible fibers are different classes in the
N\'eron--Severi group, the action of $\iota^*$ on the classes of
the components of these two reducible fibers cannot be the
identity.

If $\mathcal{E}^{(n)}$ is a fibration of type 3) on $X$, then
$\iota^*$ sends the class of the fiber $[F^{(n)}]$ to the class of
another elliptic fibration, say $[F^{(n')}]$. Hence
$\iota^*$ does not act as the identity for example on the class
$[F^{(n)}]\in NS(X)$.\endproof

\begin{corollary}\label{corollary: Fb1 and Fb2 smooth} If the fibers $(F_R)_{b_1}$ and $(F_R)_{b_2}$
of the fibration $\mathcal{E}_R:R\ra\mathbb{P}^1$ are smooth genus
1 curves, then an elliptic fibration on $X$ satisfies exactly one
of the following:
\begin{enumerate}\item is induced by a conic bundle on $R$ (and in this case the fibration is of type 1) with respect to
$\iota$); \item coincides with $\mathcal{E}_X$ (which is the
unique fibration of type 2) with respect to $\iota$); \item is of
type 3) with respect to $\iota$ and $m$ is even (where $m$ is as
\eqref{eq:  Riemann-Hurwitz}).\end{enumerate}

If at least one of the fibers $(F_R)_{b_1}$ and $(F_R)_{b_2}$ of
the fibration $\mathcal{E}_R:R\ra\mathbb{P}^1$ is a smooth genus 1
curve, then any elliptic fibration on $X$ which is of type 2)
with respect to $\iota$ coincides with $\mathcal{E}_X$.
\end{corollary}

\proof If $(F_R)_{b_1}$ and $(F_R)_{b_2}$ are smooth curves of
genus 1, then $R=\widetilde{R}$, so the notion of conic bundle and
of generalized conic bundle coincide. Moreover, we already observed
that there exists a unique elliptic fibration on $R$ and there are
no splitting genus 1 pencil on $R$ (which are not $\mathcal{E}_R$), so the unique elliptic fibration of type 2) is $\mathcal{E}_X$. It remains to prove that if there exists an
elliptic fibration on $X$ which is of the third type for $\iota$,
then $m$ is even. Let us denote by
$\mathcal{E}^{(n)}:X\ra\mathbb{P}^1$ an elliptic fibration of type 3) on $X$ and by $\mathcal{E}^{(n')}:X\ra\mathbb{P}^1$ the
elliptic fibration which is the image of $\mathcal{E}^{(n)}$ for
$\iota$. Let $F^{(n)}_t$ be a fiber of $\mathcal{E}^{(n)}$ and let
$F^{(n')}_t=\iota(F^{(n)}_t).$ The intersection points of $F^{(n)}_t$
and $F^{(n')}_{t}$ are contained in the branch locus of
$X\ra\widetilde{R}=R$, and therefore are contained in
$\pi^{-1}((F_R)_{b_1}\cap (F_R)_{b_2})$. Viceversa, each
intersection among $F^{(n)}$ and the branch locus of $X\ra R$ is a
fixed point for $\iota$ and thus is an intersection point of
$F^{(n)}_t$ and $F^{(n')}_t$. We observe that both
$\pi^{-1}((F_R)_{b_1})$ and $\pi^{-1}((F_R)_{b_2})$ are smooth
fibers of the fibration $\mathcal{E}_X:X\ra\mathbb{P}^1$, hence
$[F^{(n)}_t][\pi^{-1}((F_R)_{b_1})]=[F^{(n)}_t][\pi^{-1}((F_R)_{b_2})]=m'$
and so
$[F^{(n)}_t][\pi^{-1}((F_R)_{b_1})+\pi^{-1}((F_R)_{b_2})]=2m'=m$.

If at least one of the fibers $(F_R)_{b_1}$ and $(F_R)_{b_2}$ is
smooth, then $\iota$ fixes at least one curve of genus 1,
say $\pi^{-1}((F_R)_{b_1})$. Let $\mathcal{E}^{(n)}$ be an
elliptic fibration on $X$ which is of type 2) with respect to
$\iota$. Since $\iota$ is an involution of the basis, the fixed
locus of $\iota$ is contained in the two fibers preserved by
$\iota$. Hence $\pi^{-1}((F_R)_{b_1})$ is a smooth curve of
genus one contained in the fiber of a genus 1 fibration. This
suffices to conclude that $\pi^{-1}((F_R)_{b_1})$ is a fiber of
the fibration and hence $|\pi^{-1}((F_R)_{b_1})|$ induces the
fibration $\mathcal{E}^{(n)}$ on $X$. By definition
$\mathcal{E}_X$ is the fibration induced by
$|\pi^{-1}((F_R)_{b_1})|$, so $\mathcal{E}_X$ and
$\mathcal{E}^{(n)}$ coincide.
\endproof
\begin{corollary}\label{corollary: Fb1Fb2 smooth+iota is on NS} If the fibers $(F_R)_{b_1}$ and $(F_R)_{b_2}$ of
$\mathcal{E}_R:R\ra\mathbb{P}^1$ are smooth genus 1 curves, and
$\iota^*$ is the identity on $NS(X)$, then a fibration which is
not $\mathcal{E}_X$ is induced by a conic bundle on
$R$.

The above conditions can be realized only if
$\mathcal{R}:R\ra\mathbb{P}^1$ is an elliptic fibration without
reducible fibers.\end{corollary}

\begin{proposition}\label{prop NS in case iota is id on NS}
Let $X$ and $\iota$ be as in Corollary \ref{corollary: iota is id on NS}, i.e., $\iota$ acts as the identity on $NS(X)$. This implies that $NS(X)$ is a 2-elementary lattice and thus it is identified by its rank $r$, by its length $a$ and by the number $\delta\in\{0,1\}$ (see \cite{Nik non sympl} for the precise definition).

If $\iota$ fixes 2 curves of genus 1, then it
satisfies the more restrictive hypothesis of Corollary
\ref{corollary: Fb1Fb2 smooth+iota is on NS} and $NS(X)\simeq
U\oplus E_8(-2)$, which is the unique lattice associated to the invariant $(r,a,\delta)=(10,8,0)$. In this case $\mathcal{R}:R\ra\mathbb{P}^1$ is an elliptic fibration without reducible fibers.

If $\iota$ fixes one curve of genus 1 and possibly some rational curves, then the invariant of $NS(X)$ are $(r,a,\delta)=(r,20-r, \delta)$ with $11\leq r\leq 19$, the fiber $(F_{R})_{b_1}$ is of type $I_0$ and $(F_{R})_{b_2}$ is singular. The number and type of reducible fibers of
$\mathcal{R}:R\ra\mathbb{P}^1$ are uniquely determined by $(r,a,\delta)$ and they are: if $(r,a,\delta)=(r,20-r,1)$ or $(r,a,\delta)=(18,2,0)$, then generically the singular fibers of $\mathcal{R}$ are $I_{r-10}+(22-r)I_1$, $\rk(MW(\mathcal{R}))=19-r$ and $(F_{R})_{b_2}$ is a fiber of type $I_{r-10}$;
if $(r,a,\delta)=(14,6,0)$, then generically the singular fibers of $\mathcal{R}$ are $IV+8I_1$, $\rk(MW(\mathcal{R}))=6$ and $(F_{R})_{b_2}$ is a fiber of type $IV$.

If $\iota$ fixes only rational curves, then the invariant of $NS(X)$ are $(r,a,\delta)=(r,22-r, \delta)$ with $12\leq r\leq 20$. The number and type of reducible fibers of
$\mathcal{R}:R\ra\mathbb{P}^1$ are not uniquely determined by $(r,a,\delta)$, but one admissible choice for each $(r,a,\delta)$ is: 
if $(r,a,\delta)=(r,22-r,1)$ then generically the singular fibers of $\mathcal{R}$ are $I_{r-10}+(20-r)I_1$, $\rk(MW(\mathcal{R}))=19-r$, $(F_{R})_{b_1}$ is a fiber of type $I_1$ and $(F_{R})_{b_2}$ of type $I_{r-10}$;
if $(r,a,\delta)=(18,4,0)$, then generically the singular fibers of $\mathcal{R}$ are $2IV+4I_1$; $\rk(MW(\mathcal{R}))=4$ and both $(F_R)_{b_1}$ and $(F_{R})_{b_2}$ are fibers of type $IV$. 
\end{proposition}

\proof If $X$ is constructed by a base change of a rational
elliptic surface, then its fixed locus cannot contain curves of
genus greater then 1 (cf. also \cite{Z}). By the classification of non-symplectic involutions on K3 surfaces, we deduce that
there are the following possibilities (\cite{Nik non sympl} or \cite{Z}):\\
$i)$ $\iota$ fixes two genus 1 curves and $H^2(X,\Z)^{\iota}\simeq
U\oplus E_8(-2)$;\\
$ii)$ $\iota$ fixes one genus 1 curve and $k$ rational curve for
$1\leq k\leq 9$ and $H^2(X,\Z)^{\iota}$ is a lattice
determined by the fact that it is hyperbolic, its rank is $k+10$
and its discriminant group is $(\Z/2\Z)^{10-k}$; if $k\neq 4,8$ this determines uniquely $H^2(X,\Z)^{\iota}$, if the $k$ is $4$ or $8$ there are exactly two possibilities for $H^2(X,\Z)^{\iota}$, which depend on the value of $\delta$;\\
$iii)$ $\iota$ fixes $k$ rational curves for $2\leq k\leq 10$ and
$H^2(X,\Z)^{\iota}$ is a lattice determined by the fact
that it is hyperbolic, its rank is $10+k$ and its discriminant
group is $(\Z/2\Z)^{12-k}$; if $k\neq 8$ this determines uniquely $H^2(X,\Z)^{\iota}$, if the $k$ is $8$ there are exactly two possibilities for $H^2(X,\Z)^{\iota}$, which depend on the value of $\delta$.

Moreover, by the theory of non-symplectic involutions on K3
surfaces (see \cite{Nik non sympl}), it follows that
$H^2(X, \Z)^{\iota}\subseteq NS(X)$. If we require that $\iota$ is
the identity on $NS(X)$, we deduce that $NS(X)\simeq
H^2(X, \Z)^{\iota}$. This gives the list of lattices isomorphic
to $NS(X)$, according to the choices of the fixed locus of $\iota$
on $X$.

Let us consider the case on which $\iota$ fixes one curve of genus 1. This
implies that a curve of genus 1 is in the ramification locus of
$\pi:X\ra R$, so one fiber among $(F_R)_{b_1}$ and $(F_R)_{b_2}$
is smooth. In this case, by Corollary \ref{corollary: Fb1 and Fb2
smooth}, there exists a unique elliptic fibration on $X$ such
that $\iota$ acts as an involution of the basis and this elliptic
fibration is $\mathcal{E}_X$. Let us assume that $\iota$ fixes one
curve of genus 1 and $k$ rational curves and that $(F_R)_{b_1}$ is
a smooth curve of genus 1, then the $k$-rational curves fixed by
$\iota$ are contained in the fiber $\pi^{-1}((F_R)_{b_2})$ of the
fibration $\mathcal{E}_X$. So we know that $(F_R)_{b_2}$ satisfies the following conditions:\\ 
$(a)$ $(F_R)_{b_2}$ does not contains fibers with multiple components
(otherwise $X$ would not be a K3 surface);\\
$(b)$ $\iota$ preserves all the components of $\pi^{-1}((F_R)_{b_2})$;\\
$(c)$ $\iota$ on $\pi^{-1}((F_R)_{b_2})$ fixes $k$ curves.\\
By $(a)$, $(F_R)_{b_2}$ is of type $I_n$, $II$, $III$, $IV$. By $(b)$ it can not be of type $II$ and $III$. So, by $(c)$, if $k=4$, then $(F_R)_{b_2}$ is either of type $I_4$ or of type $IV$. If $k\neq 4$, then $(F_R)_{b_2}$ is of type $I_k$.

We already observed that if
$\iota^*$ acts as the identity on $NS(X)$, then $\mathcal{E}_R$
has no reducible fibers outside $(F_R)_{b_1}$ and $(F_R)_{b_2}$.
So all the other singular fibers of $\mathcal{E}_R$ are either of
type $I_1$ or of type $II$. Generically we can assume they are of
type $I_1$. The number of these reducible fibers depends on the
fact that the Euler characteristic of $R$ is 12 and the rank of
the Mordell--Weil group depends on the Shioda-Tate formula and on
the fact that $\rho(R)=10$.

The case $\iota$ fixes $k$ rational curves was already analyzed
in \cite{CG}. In this case $X$ admits an elliptic fibration
induced by a base change on a rational elliptic fibration if and
only if $k\geq 2$. If $k\geq 2$, the rational elliptic fibration
is not necessarily uniquely determined by $X$ and $\iota$ (indeed, there can be more than one elliptic fibration of type 2) with respect to $\iota$). In \cite[Section 5]{CG} the properties of all the rational elliptic surfaces such that a certain base change produces
$X$ are listed. Here we chose one of the admissible rational
elliptic surfaces for each K3 surface $X$.
\endproof

\section{Conic bundles on rational surfaces and elliptic fibrations on K3 surfaces}
The conic bundles are the rational fibrations on the rational elliptic surface $R$. They induce elliptic fibrations on the K3 surface $X$ and in this sense they are one of the most significative topic to relate the linear systems on $R$ to the elliptic fibrations on $X$. Since a lot is known on conic bundles on rational elliptic surfaces, here we investigate in details the relations among these fibrations on $R$ and the induced elliptic fibrations on $X$. In particular we prove that the type of singular fibers of the elliptic fibrations on $X$ induced by a conic bundle is determined by certain specific properties of the conic bundle. More precisely, we consider a conic bundle $\varphi_{|B|}:R\ra \mathbb{P}^1$ and its singular fibers. Then we consider the K3 surface $X$ which is a double cover of $R$ and the elliptic fibration $\mathcal{E}_{B}$ induced on $X$ by $\varphi_{|B|}$. Each singular fiber of the map $\varphi_{|B|}:R\ra\mathbb{P}^1$ induces a singular fiber of the fibration $\mathcal{E}_B$. In Theorem \ref{theo: reducible fibers induced by conic  bundle} we describe the relation among the reducible fibers of $\varphi_{|B|}:R\ra\mathbb{P}^1$ and of the fibration $\mathcal{E}_B$. It is also possible that some singular (reducible or not) fibers of the elliptic fibration $\mathcal{E}_B$ are not induced by singular fibers of $\varphi_{|B|}:R\ra\mathbb{P}^1$. These are called extra singular fibers and are described in Theorem \ref{theo: extra singular fibers conic  bundles}. Thus, the Theorems \ref{theo: reducible fibers induced by conic  bundle} and \ref{theo: extra singular fibers conic  bundles} together describe completely the possible configurations of the singular fibers of an elliptic fibration on the K3 surface $X$ induced by a conic bundle on the rational elliptic surface $R$.

\subsection{Conic bundles on RES}
We recall some basic information on conic bundles on rational elliptic surfaces. Every conic bundle defines a fibration $R\ra\mathbb{P}^1$ whose generic fiber is a rational curve and is such that there is always at least one reducible fiber. The singular fibers of a conic bundle are always reducible. The structure of such fibers on a rational elliptic surface is known and described in Figure \ref{cbred}. Here we briefly recall their classification, giving a self-contained proof.
\begin{proposition} Let $R$ be a relatively minimal rational elliptic surface. Let $B$ be a conic bundle on $R$. Then the singular fibers of $\varphi_{|B|}:R\ra \mathbb{P}^1$ are reducible of one of the types described in Figure  \ref{cbred}, where the multiplicity of each component is given by the number above each vertex. The empty dots correspond to $(-1)$-curves, while the full dots correspond to $(-2)$-curves. The number of components of a reducible fiber of a conic bundles is at most 10 and the number of reducible fibers is at most 9.\end{proposition}

\begin{figure}[h!]
\newcommand{\ndw}
 {\node[circle,draw, fill=white,inner sep=0pt,outer sep=0pt,minimum size=3pt]}
\begin{tikzpicture}[line cap=round,line join=round,>=triangle 45,x=1.0cm,y=1.0cm, scale=0.62]
\foreach \x in {-1,0,3,4}
 { \draw [line width=1pt] (\x+0.2,-6)-- (\x+0.8,-6); }
\draw [line width=1pt,dash pattern=on 5pt off 5pt] (1.2,-6)-- (2.8,-6);
\draw (-1.5,-6) node[anchor=north west] {{\tiny $-1$}};
\draw (-1.3,-5.3) node[anchor=north west] {{\tiny $1$}};
\draw (-0.3,-5.3) node[anchor=north west] {{\tiny $1$}};
\draw (0.7,-5.3) node[anchor=north west] {{\tiny $1$}};
\draw (2.7,-5.3) node[anchor=north west] {{\tiny $1$}};
\draw (3.7,-5.3) node[anchor=north west] {{\tiny $1$}};
\draw (4.7,-5.3) node[anchor=north west] {{\tiny $1$}};
\draw (4.5,-6) node[anchor=north west] {{\tiny $-1$}};
\draw (5,-6.7) node[anchor=north west] {};
\draw (4.4,-6.7) node[anchor=north west] {};
\fill [color=black] (4,-6) circle (3.0pt);
\ndw at (5,-6) {};
\fill [color=black] (3,-6) circle (3.0pt);
\fill [color=black] (1,-6) circle (3.0pt);
\fill [color=black] (0,-6) circle (3.0pt);
\ndw at (-1,-6) {};
\end{tikzpicture}
\hspace{0.7cm}
\begin{tikzpicture}[line cap=round,line join=round,>=triangle 45,x=1.0cm,y=1.0cm, scale=0.62]
\foreach \x in {0,3,4}
 { \draw [line width=1pt] (\x+0.2,-6)-- (\x+0.8,-6); }
\draw [line width=1pt] (-.8,-5.2)-- (-.2,-5.8);
\draw [line width=1pt] (-.8,-6.8)-- (-.2,-6.2);
\draw [line width=1pt,dash pattern=on 5pt off 5pt] (1.2,-6)-- (2.8,-6);
\draw (4.7,-5.3) node[anchor=north west] {{\tiny $2$}};
\draw (4.6,-6) node[anchor=north west] {{\tiny $-1$}};
\draw (3.7,-5.3) node[anchor=north west] {{\tiny $2$}};
\draw (2.7,-5.3) node[anchor=north west] {{\tiny $2$}};
\draw (0.7,-5.3) node[anchor=north west] {{\tiny $2$}};
\draw (-0.2,-5.3) node[anchor=north west] {{\tiny $2$}};
\draw (-1.3,-4.3) node[anchor=north west] {{\tiny $1$}};
\draw (-1.3,-6.3) node[anchor=north west] {{\tiny $1$}};
\fill [color=black] (4,-6) circle (3.0pt);
\ndw at (5,-6) {};
\fill [color=black] (3,-6) circle (3.0pt);
\fill [color=black] (1,-6) circle (3.0pt);
\fill [color=black] (0,-6) circle (3.0pt);
\fill [color=black] (-1,-5) circle (3.0pt);
\fill [color=black] (-1,-7) circle (3.0pt);
\end{tikzpicture}
 \caption{Reducible fibers of conic bundles}\label{cbred}
\end{figure}
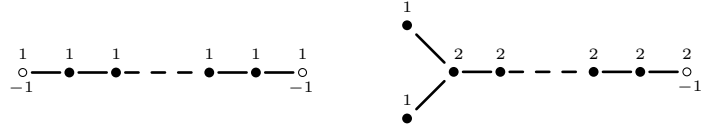
\proof The generic fiber of $\varphi_{|B|}:R\ra \mathbb{P}^1$ is a curve of genus 0, so if it is singular, it is reducible and its components are negative curves. The negative curves on $R$ are either $(-1)$-curves (sections of $\mathcal{E}_R$) or $(-2)$-curves (components of reducible fibers of $\mathcal{E}_R$). 
We recall the definition of  generalized $(-1)$-curve and generalized $(-2)$-curve: a connected reduced (not necessarily irreducible) curve $E$ is a generalized $(-1)$-curve (resp. $(-2)$-curve) if there exists a morphism to a smooth surface contracting $E$ to a point (resp. if the intersection graph on its irreducible components is a Dynkin diagram). It follows that a generalized $(-1)$-curve is either a $(-1)$-curve or a chain of $(-2)$-curves with one $(-1)$-curve on one extreme. 
Let $D$ be a reducible fiber of the conic bundle $\varphi_{|B|}:R\ra\mathbb{P}^1$. Let $D^{red}$ be the reduced sum of its integral component and let $D_1$ be a component of $D^{red}$ such that $D^{red}-D_1$ is still connected. We prove that for every choice of $D_1$ in $D^{red}$, $D^{red}-D_1$ is either a generalized $(-1)$-curve or a generalized $(-2)$-curve. Since we know the admissible structure of these generalized curve, we deduce the admissible configuration of components of the reducible fibers of a conic bundle. 

So, let $D^{red}$ and $D_1$ as above. Then $D^{red}-D_1$ is a connected and reduced divisor of X such that the intersection form on the irreducible components of its support is negative definite. If $D^{red}-D_1$ contains only $(-2)$-curves, then these curves are all contained in the same reducible fiber of the elliptic fibration on $R$, and hence they form a Dynkin diagram. In this case $D^{red}-D_1$ is a generalized $(-2)$-curve. If $D^{red}-D_1$ contains exactly one $(-1)$-curve, $P$, then this has to be on one extreme of the chain of curves in $D^{red}-D_1$. Indeed if there where two (generalized) $(-2)$-curves $G_1$ and $G_2$ which both intersects $P$, then $(2P+G_1+G_2)^2=-4-2-2+4PG_1+4PG_2+2G_1G_2\geq 0$, but the intersection form on the support of $D^{red}-D_1$ is negative definite. Suppose now that there exists at least two $(-1)$-curves $P_1$ and $P_2$ in the support of $D^{red}-D_1$. In this case either they intersect, but this gives a contradiction, because $(P_1+P_2)^2=-1-1+2P_1P_2\geq 0$, or there exists a chain of $(-2)$-curves $G$ such that $G$ intersects both $P_1$ and $P_2$. But also in this case we have a contradiction, since $G^2=-2$ and so $(P_1+P_2+G)^2=-1-1-2+2P_1G+2P_2G\geq 0$.

Hence for every $D_1$ in the support of $D^{red}$ such that $D^{red}-D_1$ is connected, $D^{red}-D_1$ is either a Dynkin diagram or a chain of $(-2)$-curves with one $(-1)$-curve on one extreme. Moreover we know that $D$ is the fiber of a conic bundle, and so $D(-K_R)=2$ and thus either in the support of $D$ there are two $(-1)$-curves or in the support of $D$ there  is a $(-1)$-curve with multiplicity 2. In the first case, one is forced to chose $D_1$ to be a $(-1)$-curve (otherwise $D^{red}-D_1$ contains two $(-1)$-curves) and one obtains that $D^{red}-D_1$ is a generalized $(-1)$-curve, and this gives the first fiber in Figure \ref{cbred}. In the second case choosing a $(-2)$-curve as $D_1$ one obtains that $D^{red}-D_1$ is a generalized $(-1)$-curve. In particular this implies that either $D$ is the second fiber in Figure \ref{cbred} with exactly three components or the $(-1)$-curve in $D$ is on an extreme of $D$. So the $(-1)$-curve can be chosen as $D_1$. Choosing the $(-1)$-curve as $D_1$ one obtains a Dynkin diagram. The diagram of type $E_m$ has to be excluded, indeed if $D^{red}-D_1$ is a diagram of type $E_m$, $m=6,7,8$, then it is always possible choose a $(-2)$ curve $D_2$ in $D^{red}$ such that $D^{red}-D_2$ is connected and contains a $(-1)$-curve (and thus is not a generalized $(-2)$-curve), but it is not a generalized $(-1)$-curve, which is impossible.

The bound of the number of the reducible components and of the number of their components follows by the fact that the Picard number of $R$ is 10.
\endproof

\begin{definition} We say that a reducible fiber of a conic bundle is of type $A_n$ if it has $n$ components with the intersection properties which appear on the left of the Figure \ref{cbred} and we say that a reducible fiber of a conic bundle is of type $D_m$ if it has $m$ components with the intersection properties which appear on the right of the Figure \ref{cbred}. 
\end{definition}

\subsection{Elliptic fibrations induced on double cover of RES}
Given a conic bundle on $R$, it induces an elliptic fibration on $X$. The aim of this section is to determine the reducible fibers of the elliptic fibration on $X$ by the knowledge of the reducible fibers on $R$.  
\begin{theorem}\label{theo: reducible fibers induced by conic bundle}
Let  $|B|$ be a conic bundle on a rational elliptic surface $R$. It admits at least one reducible fiber, which is either of type $A_n$ or of type $D_m$. Assume that there is a $(-2)$ curve $S$ on the surface $R$ which is a section of $\varphi_{|B|}:R\ra\mathbb{P}^1$. Then $n\leq 9$ and $m\leq 9$.

Let $X$ be a K3 surface obtained by a base change on $R\ra\mathbb{P}^1$.
Then $\varphi_{|B|}:R\ra\mathbb{P}^1$ induces an elliptic fibration (with section!) on $X$. 

The correspondence between the reducible fibers of $|B|$ and the ones induced on the elliptic fibration on $X$ is presented in Table \ref{table: red fibers conic bundle and elliptic}.

\begin{table}
\begin{tabular}{c|c||c}
&Fiber conic bundle&Fiber elliptic fibration\\
\hline $ $\\
\hline
$A_n$& no $(-2)$-curves in the branch locus&
       $I_{2n-2}$\\
\hline
$A_n$& \begin{tabular}{c} the $(-2)$-curves are in the branch locus\\ components of a fiber of type $I_k$\end{tabular}&
       $I_{2n-6}^*$\\
\hline
$A_3$& \begin{tabular}{c} the $(-2)$-curve is in the branch locus\\ components of a fiber of type $III$\end{tabular}       
       &$I_{1}^*$\\
\hline
$A_3$&\begin{tabular}{c} the $(-2)$-curve is in the branch locus\\ components of a fiber of type $IV$\end{tabular}&
          $I_{2}^*$\\
\hline
$A_4$& \begin{tabular}{c} the $(-2)$-curves are in the branch locus\\ components of a fiber of type $IV$\end{tabular}&
       $III^*$\\
\hline
  $D_m$&
\begin{tabular}{c} no $(-2)$-curves in the branch curves\end{tabular}&
$I_{2m-6}^*$\\
\hline      
$D_4$&
\begin{tabular}{c} the $(-2)$-curves are in the branch locus \\
components of a fiber of type $I_k$\end{tabular}& $III^*$\\

\hline      
$D_3$&
\begin{tabular}{c} one $(-2)$-curve not in the branch locus \\
one $(-2)$ curve in the branch locus\\
component of a fiber of type $I_k$\end{tabular}& $I_1^*$\\

\hline      
$D_3$&
\begin{tabular}{c} one $(-2)$-curve not in the branch locus \\
one $(-2)$ curve in the branch locus\\
component of a fiber of type $III$\end{tabular}& $I_2^*$\\

\hline      
$D_3$&
\begin{tabular}{c} one $(-2)$-curve not in the branch locus \\
one $(-2)$ curve in the branch locus\\
component of a fiber of type $IV$\end{tabular}& $I_3^*$\\
\hline      
$D_3$&
\begin{tabular}{c} both the $(-2)$-curves in the branch locus \\
one component of a fiber of type $I_k$\\
one component of a fiber of type $I_h$\\
\end{tabular}& $I_2^*$\\

\hline      
$D_3$&
\begin{tabular}{c} both the $(-2)$-curves in the branch locus \\
one component of a fiber of type $I_k$\\
one component of a fiber of type $III$\\
\end{tabular}& $I_3^*$\\

\hline      
$D_3$&
\begin{tabular}{c} both the $(-2)$-curves in the branch locus \\
one component of a fiber of type $I_k$\\
one component of a fiber of type $IV$\\
\end{tabular}& $I_4^*$\\

\hline      
$D_3$&
\begin{tabular}{c} both the $(-2)$-curves in the branch locus \\
both components of a fiber of type $III$\\
\end{tabular}& $I_4^*$\\

\hline      
$D_3$&
\begin{tabular}{c} both the $(-2)$-curves in the branch locus \\
one component of a fiber of type $III$\\
one component of a fiber of type $IV$\\
\end{tabular}& $I_5^*$\\

\hline      
$D_3$&
\begin{tabular}{c} both the $(-2)$-curves in the branch locus \\
both components of a fiber of type $IV$\\
\end{tabular}& $I_6^*$
\end{tabular}\caption{Reducible fibers of conic bundles and corresponding reducible fibers of the elliptic fibration on $X$}\label{table: red fibers conic bundle and elliptic}\end{table}
\end{theorem}

\proof The bound on $n$ and $m$ follows by the fact that the Picard number of $R$ is 10 and the class of the section $S$ is surely independent from the classes of the irreducible components of reducible fibers of the rational fibration $\varphi_{|B|}:R\ra \mathbb{P}^1$.

We already saw the structure of the reducible fiber of a conic bundle on a rational elliptic surface and we proved in Proposition \ref{prop: (generalized) conic bundles and splitting genus 1 pencil induces genus 1 fibration} that each conic bundle $|B|$ on $R$ induces a genus 1 fibration $\mathcal{E}_B:X\ra\mathbb{P}^1$. Let us assume that $S$ is a section of a conic bundle on $R$ and moreover that it is a $(-2)$ curve on $R$. Since $S$ is a section of $\varphi_{|B|}:R\ra\mathbb{P}^1$, it meets each fiber in one point. The fibers of $\mathcal{E}_B$ are the pull back of the fibers of the conic bundle. If the curve $S$ is in the branch locus of the $2:1$ cover $X\ra R$, then the intersection point between $S$ and a fiber of $\varphi_{|B|}:R\ra\mathbb{P}^1$ corresponds to one point of intersection between the inverse image of $S$ and a fiber of $\mathcal{E}_B$. So the inverse image of $S$ is a section of $\mathcal{E}_B$. Viceversa, if $S$ is not in the branch locus, then it has a trivial intersection with the branch locus, since the branch locus consists of two fibers of the elliptic fibration $R\ra\mathbb{P}^1$ and $S$ is necessarily a component of another reducible fiber. In this case the inverse image of the point of intersection between $S$ and a fiber of the conic bundle $|B|$ consists of two distinct points. The curve $S$ splits into two disjoint rational curves in the double cover $X$ (because it does not intersect the branch locus and it is rational) and each curve intersects the fibers of $\mathcal{E}_B$ in exactly one point. Thus the inverse image of $S$ consists of two disjoint sections of $\mathcal{E}_B$ the genus 1 fibration on $X$. 
In both the cases $\mathcal{E}_B:X\ra\mathbb{P}^1$ is in fact an elliptic fibration.
Now we consider the types of reducible fibers induced on the elliptic fibration $\mathcal{E}_B:X\ra\mathbb{P}^1$ by the reducible fibers of the conic bundle $|B|$.

We describe now the inverse image on $X$ of the negative curves on $R$. In particular we are interested in the components of the reducible fibers of $|B|$.

Each $(-1)$-curve is covered by a unique rational curve on $X$.

First we assume that all the $(-2)$-curves in the reducible fibers of the conic bundle are not in the branch locus of the double cover $X\ra R$, so that each $(-2)$-curve is covered by two disjoint rational curves on $X$. If the reducible fiber of the conic bundle is of type $A_n$, then it induces a fiber on the elliptic fibration on $X$ which is of type $I_{2n-2}$. If the reducible fiber of the conic bundle is of type $D_m$, then it induces a fiber of type $I_{2m-6}^*$ on $\mathcal{E}_B$.  

Now we assume that some of the $(-2)$-curves in the reducible fibers of the conic bundles are branch curves of the double cover $X\ra R$. This implies that the $(-2)$-curves appearing in the reducible fibers of the conic bundle are contained in the branch fibers of the fibration $\mathcal{E}_R$, i.e. they are contained in $(F_R)_{b_1}\cup (F_R)_{b_2}$.

The components of the reducible fibers of a conic bundle which are of type $A_n$ are two sections and $n-2$ $(-2)$-curves which are all contained in the same fiber which is necessarily of type $I_n$, $III$, $IV$. So either all the $(-2)$-components are in the branch locus of the cover $X\ra R$ or none of the $(-2)$-components are in the branch locus. The same is true for the reducible fibers of a conic bundle which are of type $D_m$, with $m>3$. For the conic bundles of type $D_3$ there are more possibilities, indeed the two $(-2)$-curves appearing are contained in different fibers of the rational elliptic fibration, so it is possible that both of them are in the branch locus, or only one of them is. 

Let us suppose that there are two $(-2)$-curves in the branch locus which are contained in a fiber of type $I_k$, $k>2$. In this case, in order to construct $X$ we first blow up $R$ in their intersection point introducing a $(-1)$-curve, which is not in the branch locus and then consider the double cover. This gives three $(-2)$-curves on $X$: two of them are mapped $1:1$ to the $(-2)$-curves we are considering in $R$ and are disjoint in $X$, the third is a $2:1$ cover of the exceptional divisor of the blow up of $R$ and meets both of the other two curves, each in one point.

We conclude that if the $(-2)$-curves of a reducible fiber of type $A_n$ of the conic bundle are contained in a fiber of type $I_k$ of the rational elliptic fibration, then the reducible fiber of the conic bundle induces a fiber of type $I_{2n-6}^*$ on $\mathcal{E}_B$.

If the $(-2)$-curves of a reducible fiber of type $D_m$, with $m\geq 4$, of the conic bundle are contained in a fiber of type $I_k$ of the rational elliptic fibration, then the reducible fiber of type $D_m$ of the conic bundle is of type $D_4$, i.e. $m=4$. Such a fiber induces a fiber of type $III^*$ on $\mathcal{E}_B$.

If there are $(-2)$-curves components of a reducible fiber of the conic bundle which are contained in the branch locus of $X\ra R$ but they are not contained in a fiber of type $I_k$ of $\mathcal{E}_R$, the they are components of fibers of type either $III$ or $IV$. In this case their inverse image in $X$ is slightly different, but it can be computed in a similar way and is given in Table \ref{table: rational curves on R and X}.  

Summarizing the discussion above, we have that a $(-2)$-curve which is a component of a reducible fiber of a conic bundle on $R$ corresponds to a certain configuration of $(-2)$-curves on $X$, which is presented in the following Table \ref{table: rational curves on R and X}. In the table the symbol $\bullet$ always corresponds to a $(-2)$-curve. The letter $e$ near a point means that the point corresponds to a $(-2)$-curves on $X$ which is the cover of an exceptional curves in the blow up $\widetilde{R}\ra R$. Considering all the possible configurations of $(-2)$-curves which are components of reducible fibers of a conic bundle on $R$, we obtain the statement.

\begin{table}
\begin{tabular}{|c|c|c|}
\hline
On $R$& Location &On $X$\\
\hline
$\begin{array}{c}\bullet\\
-2\end{array}$&not in the branch&$\begin{array}{ccc}\bullet&&\bullet\\
-2&&-2\end{array}$\\
\hline
$\begin{array}{c}\bullet\\
-2\end{array}$&\begin{tabular}{c}in the branch\\ component of fiber $I_k$\end{tabular}&$$
\xymatrix{
         \bullet_e\ar@{-}[r] &\bullet \ar@{-}[r]&\bullet_{e'}}$$\\
\hline
$\begin{array}{c}\bullet\\
-2\end{array}$&\begin{tabular}{c}in the branch\\ component of fiber $III$\end{tabular}&$$
\xymatrix{
         \bullet_{e'}\ar@{-}[dr]\\
          &\bullet_{e} \ar@{-}[r]&\bullet\\
          \bullet_{e'}\ar@{-}[ur]\\
          }$$\\
\hline
$\begin{array}{c}\bullet\\
-2\end{array}$&\begin{tabular}{c}in the branch\\ component of fiber $IV$\end{tabular}&$$
\xymatrix{
         \bullet_{e'}\ar@{-}[dr]\\
          &\bullet_{e} \ar@{-}[r]&\bullet_{e'''}\ar@{-}[r]&\bullet\\
          \bullet_{e''}\ar@{-}[ur]\\
          }$$\\
\hline
\end{tabular}\caption{Rational curves on $R$ and on $X$}\label{table: rational curves on R and X}
\end{table}\endproof

\begin{definition}Let us consider an elliptic fibration $\mathcal{E}_B:X\ra\mathbb{P}^1$ induced by a conic bundle $|B|$ on $R$. We say that a singular fiber of $\mathcal{E}_B$ is an extra singular fiber (with respect to $|B|$) if it is not induced by a reducible fiber of the conic bundle $|B|$.
\end{definition}
\begin{theorem}\label{theo: extra singular fibers conic bundles}
Let $|B|$ be a conic bundle on $R$ and $\mathcal{E}_B$ the elliptic fibration induced by $|B|$ on $X$. Then the extra singular fibers of $\mathcal{E}_B$ are induced by the curves $B'\in |B|$ such that the morphism $B'\rightarrow \mathbb{P}^1$ given by composing the inclusion map $B'\subset R$ with the elliptic fibration map on $R$ share a ramification with the base change map $f:\mathbb{P}^1\rightarrow \mathbb{P}^1$ ($f$ is as in \eqref{diagram fiber  product}). Moreover, all irreducible singular fibers of $\mathcal{E}_B$ are also of the above form.

The type of the extra singular fibers of $\mathcal{E}_B$ are listed in Table \ref{table: extra singular fibers}.
\begin{table}
\begin{tabular}{c|c||c|c||c}
$(F_R)_{b_1}$&$(F_R)_{b_2}$&$B\cap (F_R)_{b_1}$&$B\cap (F_R)_{b_2}$&Fiber on $\mathcal{E}_B$\\
\hline $ $\\
\hline
Any&Any& 2 distinct points& tangent&$I_1$\\
\hline
Any&$I_n$& 2 distinct points& singular point&$I_2$\\
\hline
Any&$II$& 2 distinct points& singular point&$I_3$\\
\hline
Any&$III$& 2 distinct points& singular point&$I_4$\\
\hline
Any&Any& tangent & tangent&$I_2$\\
\hline
Any&$I_n$& tangent& singular point&$I_3$\\
\hline
Any&$II$&  tangent& singular point&$I_4$\\
\hline
Any&$III$&  tangent& singular point&$I_5$\\
\hline
$I_n$&$I_n$& singular point& singular point&$I_4$\\
\hline
$I_n$&$II$& singular point& singular point&$I_5$\\
\hline
$I_n$&$III$& singular point& singular point&$I_6$\\
\hline
$II$&$II$& singular point& singular point&$I_6$\\
\hline
$II$&$III$& singular point& singular point&$I_7$\\
\hline
$III$&$III$& singular point& singular point&$I_8$\\
 
\end{tabular}\caption{Extra singular fibers of $\mathcal{E}_B$}\label{table: extra singular fibers}\end{table}
\end{theorem}

\begin{proof}
Let $F'$ be an extra singular fiber of $\mathcal{E}_B$. Then $F'$ is singular but is induced by a smooth curve $B'\in |B|$. Since $F'$ is isomorphic to the blow up of the fiber product of two smooth curves $B' \times_{\mathbb{P}^1}\mathbb{P}^1$, it is singular if and only if $B'\rightarrow \mathbb{P}^1$ and $f:\mathbb{P}^1\rightarrow \mathbb{P}^1$ share a ramification point. Moreover, since the singular fibers of the conic bundle $|B|$ are reducible, they yield reducible singular fibers on the induced elliptic fibration. Therefore the irreducible singular fibers are of the former type.

Now we prove that the type of the singular fibers is the one given in Table \ref{table:  extra singular fibers}.

Since $B'\ra\mathbb{P}^1$ shares at least one branch point with the base change $f:\mathbb{P}^1\ra\mathbb{P}^1$, associated to $X\ra R$, $B'$ meets at least one of the branch fibers $(F_R)_{b_1}$, $(F_R)_{b_2}$ in a point with multiplicity 2. Since $B'$ is a fiber of a conic bundle on $R$, it is a bisection of the fibration $\mathcal{E}_R:R\ra\mathbb{P}^1$, hence the unique possibilities (up to interchanging $b_1$ with $b_2$) are the followings:
\begin{enumerate}\item $B'$ meets $(F_R)_{b_1}$ in two disjoint points and meets $(F_R)_{b_2}$ in a smooth point with multiplicity 2 (i.e. $B_t$ is tangent to $(F_R)_{b_2}$);
\item $B'$ meets $(F_R)_{b_1}$ in two disjoint points and meets $(F_R)_{b_2}$ in a singular double point;
\item  $B'$ meets $(F_R)_{b_1}$ in a smooth point with multiplicity 2 and meets also  $(F_R)_{b_2}$ in a smooth point with multiplicity 2 (i.e. $B'$ is tangent both to $(F_R)_{b_1}$ and to $(F_R)_{b_2}$);
\item $B'$ meets $(F_R)_{b_1}$ in a smooth point with multiplicity 2 and meets $(F_R)_{b_2}$ in a singular double point;
\item $B'$ meets $(F_R)_{b_1}$ in a singular double point and meets $(F_R)_{b_2}$ in a singular double point.
\end{enumerate}

In the first two cases $B'$ does not split in the double cover $X\ra R$, in the others it splits.
Moreover, we observe that if $B'$ passes through a singular point of $(F_R)_{b_i}$, then this point is a double point, and thus $(F_R)_{b_i}$ is not of type $IV$.

Now we give the details of some of the lines of the Table \ref{table: extra singular fibers}, the others are similar. We denote by $F'$ the fiber of $\mathcal{E}_B$ induced by $B'$

If $B'$ meets $(F_R)_{b_1}$ in two disjoint points and is tangent to $(F_R)_{b_2}$, $\beta_2^*(B_t)$ is isomorphic to $B'$ and its double cover $F'\ra\beta_2^*(B')$ is a connected curve, which is singular in exactly the point $\pi^{-1}(\beta_2^{-1}(B'\cap (F_R)_{b_2})$. Hence $F'$ is a fiber of type $I_1$.

If $B'$ meets $(F_R)_{b_1}$ in two disjoint points and meets $(F_R)_{b_2}$ in a singular double point, $\beta_2^*(B')$ consists of the strict transform of $B'$ and some exceptional divisors. The number and the intersection properties of the exceptional divisors depends on the singularities of $(F_R)_{b_2}$ in $b:=(F_R)_{b_2}\cap B'$. We know that $(F_R)_{b_2}$ is one of the following fibers: $I_n$, $II$, $III$.  If $(F_R)_{b_2}$ is a fiber of type $I_n$
then $\beta_2$ introduces one exceptional divisor on $b$, say $E_b:=\beta_2^{-1}(b)$. The strict transform of $B'$, $\widetilde{B'}$, meets $E_b$ in one point. Hence $\beta_2^*(B')$ consists of two smooth curves meeting in one point.  We recall that $E_b$ is not contained in the branch locus of $\pi:X\ra\widetilde{R}$ and its double cover consists of a connected irreducible rational curve. So $\pi^{-1}(E_b\cap \widetilde{B'})$ consists of two points. Hence $F'$ is given by two rational curves meeting in two points. So $F'$ is a fiber of type $I_2$.
If $(F_R)_{b_2}$ is a fiber of type $II$, then $\beta_2$ introduces one exceptional divisor on $b$, say $E_b:=\beta_2^{-1}(b)$. The strict transform of $B'$, $\widetilde{B'}$, meets $E_b$ in one point. Hence $\beta_2^*(B')$ consists of two smooth curves meeting in one point. The curve $E_b$ is not contained in the branch locus of $\pi:X\ra\widetilde{R}$, and it splits in two curves $E'_b$ and $E''_b$. So  $\pi^{-1}(E_b\cap \widetilde{B'})$ consists of two points (one on $E'_b$ and the other on $E''_b$). Hence $F'$ is given by three rational curves (the inverse image of $\widetilde{B'}$, $E'_b$ and $E''_b$) meeting as a triangle (the inverse image of $\widetilde{B'}$ meets both $E'_b$ and $E''_b$). So $F'$ is of type $I_3$.
If $(F_R)_{b_2}$ is a fiber of type $III$, then $\beta_2$ introduces two exceptional divisors. One of them splits in the double cover, the other no, hence $F'$ is a fiber of type $I_4$.

If $B'$ is tangent both to $(F_R)_{b_1}$ and to $(F_R)_{b_2}$, $\beta_2^{-1}(B')\simeq \widetilde{B'}\simeq B'$. On the other hand $\widetilde{B'}$ meets the branch locus in two points, each with even multiplicity. So $\widetilde{B'}$ splits in the double cover and $\pi^{-1}(\widetilde{B'})$ consists of two rational curves $B''$ and $B'''$. These two curves meets in two points so the fiber $F'$ is a fiber of type $I_2$.
\end{proof}

\begin{rem}{\rm The reducible fibers of the elliptic fibration on $X$ induced by a conic bundle on $R$ are of the following types:\\ {\bf induced by reducible fibers of the conic bundles:} $I_{2n}$ with $n\leq 9$, $I_{2m}^*$ for $0\leq m\leq 6$, $I^*_{2l+1}$ with $l\leq 2$ and $III^*$;\\
{\bf extra singular fibers:} $I_n$ with $n=1,\ldots,8 $.

In particular no fibers of type $II^*$, $IV^*$, $II$, $III$, $IV$ appear. The fibers of type $I_n$ do not appear for odd $n>7$ and the fibers of type $I^*_m$ do not appear for odd $m>5$ and even $m>12$.}\end{rem}

\begin{rem}{\rm
If the base change map which induces $X$ is ramified only at smooth fibers, then any elliptic fibration on $X$ induced by a conic bundle on $R$ has at most 6 irreducible singular fibers. Indeed, the irreducible singular fibers are extra singular fibers. By the theorem above, these correspond to conics such that the map to the base shares a ramification point with the base change map, and are therefore tangent to the fibers above the ramification. Since these are smooth elliptic curves, these tangent conics correspond to non-trivial 2-torsion points on the elliptic curve, which are at most 3 points. Since the base change map is ramified at two fibers, there are at most 6 extra singular fibers, and therefore at most 6 irreducible singular fibers.}
\end{rem}

\subsection{The elliptic fibrations on a very particular K3 surface}\label{subsec: example K3 double cover R211}

Let $R_{211}$ be a rational elliptic surface with one reducible fiber of type $II^*$ and two irreducible singular fibers of type $I_1$. We suppose that the $II^*$ fiber occurs at $s=\infty$. Then $R_{211}$ belongs to the following family:
\[
y^2=x^3+ ax+s+b,
\]  
where $a$ and $b$ are the two parameters of the family and the surface $R_{211}$ corresponds to an elliptic curve over the function field $\mathbb{C}(s)$. Moreover, $R_{211}$ is isomorphic to the blow up of $\mathbb{P}^2$ at the point $P=(0:1:0)$ and other 8 infinitely near points.

Let $X_{211}$ be a K3 surface obtained by a degree two base change of $R_{211}$ ramified at two smooth fibers. Then $X_{211}$ is naturally endowed with an elliptic fibration whose singular fibers are of types $2II^*+ 4I_1$. This K3 surface is very well known, its N\'eron--Severi group is isometric to $U\oplus E_8\oplus E_8$, its automorphism group is finite, \cite[Theorem]{K}, the number of $(-2)$-curves is finite and their intersection properties are known. Thanks to this, also the classification of all the elliptic fibrations on $X_{211}$ is known (and it is based on the possible reducible fibers one can construct with the finite $(-2)$-curves on the surface). 
To be more precise, the K3 surface $X_{211}$ admits exactly two elliptic fibrations. One has $2II^*+4I_1$ as singular fibers and is the elliptic fibration induced by the elliptic fibration on $R$. The second has $I_{12}^*+6I_1$ as singular fibers. 

Since $R_{211}$ is the blow up of $\mathbb{P}^2$ in a unique point $P$, $R_{211}$ has only one conic bundle, which is given by the linear system of lines through $P$. This conic bundle has a unique reducible fiber, which corresponds to the line $l$ of the pencil which has the direction identified by the infinitely near points of the pencil of cubics defining $R_{211}$. This reducible fiber is of type $D_9$: indeed in order to construct $R_{211}$ one blows up $P$ 9 times introducing a generalized $(-1)$-curve, called $E$, with nine components. The components of the fiber of type $II^*$ of the elliptic fibration on $R_{211}$ are the eight $(-2)$-curves in the support of $E$ and the strict transform of the line $l$. The extreme curves of $E$ are a $(-1)$-curve and the $(-2)$-curve which is the strict transform of the first blow up of $P$. This curve is a section of the rational fibration of the conic bundle, indeed it separates the lines through $P$. Hence this component is not contained in the reducible fibers of the conic bundle. So the components of the reducible fiber of the conic bundle are eight curves in the support of $E$ (seven $(-2)$-curves and the $(-1)$-curve) and the strict transform of the line $l$. This corresponds to a reducible fiber of type $D_9$. By the construction of $X_{211}$, the components of this reducible fiber of the conic bundle are not contained in the branch locus of the double cover $X_{211}\ra R_{211}$. By Table \ref{table: red fibers conic bundle and elliptic},  the induced reducible fiber on the elliptic fibration on $X$ is of type $I_{12}^*$. There are three lines through $P$ which are tangent to a specific cubic. Since $X$ is obtained as double cover of $R$ branched on two smooth fibers (and so on curves which correspond to two smooth cubics in $\mathbb{P}^2$), there are six extra singular fibers on the elliptic fibration on $X$. All of them are of type $I_1$, by Table \ref{table:  extra singular fibers}.

We conclude that the elliptic fibrations on the surface $X_{211}$ are either induced by conic bundles or by the elliptic fibration on the $R_{211}$.

\section{Examples}\label{sec: example}
The aim of this section is to describe explicitly examples of the different types of elliptic fibrations introduced throughout the text. We consider a specific rational elliptic surface, denoted in this section by $R$, and a specific double cover of it, denoted in this section by $X$. Then we explicitly exhibit on $X$: an elliptic fibration induced by a conic bundle on $R$ (see Section \ref{subsec: example conic bunldes}); an elliptic fibration induced by a generalized conic bundle on $\widetilde{R}$ (see Section \ref{subsec: examples generalized conic  bundle}); an elliptic fibration of type 2), induced by a splitting genus 1 pencil on $\widetilde{R}$ (see Section \ref{subsec: example splitting genus 1  pencil}); finally an elliptic fibration of type 3), induced by a non-complete linear system on $\widetilde{R}$ (see Section \ref{subsec:  example type 3) fibration}).

\subsection{The surfaces $R$ and $X$}\label{subsec: example R and X}
Let us consider the rational elliptic surface which is the blow up of $\mathbb{P}^2_{(x:y:z)}$ in the base points of the pencil of cubics 
$$\mathcal{P}_{\alpha}:=V\left(xyz+\lambda(x-y)(y-z)(x-\alpha z)\right).$$
This is a 1-dimensional space of pencils of cubics, and for a generic choice of $\alpha$ the elliptic fibration $\mathcal{E}_R$ associated to this pencil of cubics has 4 singular fibers: one fiber of type $I_6$ on $\lambda=0$, one fiber of type $I_3$ on $\lambda=\infty$ and three fibers of type $I_1$. The Mordell--Weil group of the fibration is $\Z\times\Z/3\Z$. Indeed the rank of the Mordell--Weil group of the fibration can be computed directly by the Shioda-Tate formula and the presence of a section of order 3 follows directly by the following choice of a $\Z$-basis of the N\'eron--Severi group of $R$: Let us denote by $l_1$, $l_2$, $l_3$, $m_1$, $m_2$ and $m_3$ the lines $x=0$, $y=0$, $z=0$, $x=y$, $y=z$ and $x=\alpha z$ respectively. Let $P_{ij}=l_i\cap l_j$, $1\leq i<j\leq 3$. The points $P_{ij}$ are base points of the pencil $\mathcal{P}_{\alpha}$ and each of them has an infinitely near point, which correspond to a tangent direction for the cubics of the pencil. More precisely, all the cubics of the pencil pass through $P_{12}$, $P_{23}$ and $P_{13}$ with tangent directions $m_1$, $m_2$ and $m_3$ respectively. The other 3 base points of the pencil are $Q_{110}=m_1\cap l_3$, $Q_{011}=m_2\cap l_3$ and $Q_{\alpha0-1}=m_3\cap l_3$. We denote by $E_{ij}$ and $F_{ij}$ the exceptional divisors on $R$ which are mapped to the point $P_{ij}$ and are such that $E_{ij}^2=-2$, $F_{ij}^2=-1$, $E_{ij}F_{ij}=1$. Moreover we denote by $E_{011}$, $E_{110}$ and $E_{\alpha0-1}$ the exceptional divisors over the points $Q_{011}$, $Q_{110}$ and $Q_{\alpha0-1}$ respectively.
Then, a $\Z$-basis of the N\'eron--Severi group of $R$ is given by $F_{12}$, $F_{13}$, $E_{011}$, $E_{12}$, $\widetilde{l_1}$, $E_{13}$, $\widetilde{l_3}$, $E_{23}$, $\widetilde{m_2}$, $\widetilde{m_3}$, where $\widetilde{l_i}$ and $\widetilde{m_j}$ correspond to the strict transform of $l_i$ and $m_j$ on $R$.

We observe that the classes $\widetilde{l_i}$ and $E_{ij}$ are the components of the fiber of type $I_6$, the classes $\widetilde{m_j}$ are the components of the fiber of type $I_3$. Chosen $F_{12}$ to be the zero section of the fibration, the height formula (\cite[11.8]{SS}) implies that $F_{13}$ is a 3-torsion section and $E_{011}$ is a section of infinite order.

Let $X$ be the K3 surface obtained as double cover of $R$ branched along a fiber of type $I_6$ and a fiber of type $I_1$. The elliptic fibration $\mathcal{E}_X$ induced on $X$ by $\mathcal{E}_R$ has $I_{12}+2I_3+I_2+4I_1$ as singular fibers and $MW(\mathcal{E}_X)\supset MW(\mathcal{E}_R)=\Z\times \Z/3\Z$. Since the $\rk(MW(\mathcal{E}_X))\geq \rk(MW(\mathcal{E}_R))=1$, and the rank of the trivial lattice of $\mathcal{E}_X$ is 2+11+2+2+1=18, it follows that the Picard number of $X$ is at least 19. On the other hand, the construction of $X$ depends on 1 parameter (the parameter $\alpha$ on which depends the construction of $R$). So the Picard number of $X$ is at most 20-1=19. It follows that the Picard number of $X$ is exactly 19 and that the rank of the Mordell-Weil group of $\mathcal{E}_X$ is exactly 1.

In order to construct a smooth model of $X$ one has to blow up each of the intersection points between the components of the fiber of type $I_6$ on $R$ and the node of the branch fiber of type $I_1$ on $R$, obtaining $\widetilde{R}$. This means that $\beta_1\circ\beta_2:\widetilde{R}\ra\mathbb{P}^2$ blows up each point $P_{ij}$ 4 times: first one introduces an exceptional divisor which has multiplicity 3, then one has to blow up the intersection point between this exceptional component and the strict transforms of the lines $l_i$, $l_j$ and $m_k$, for a certain $k\in\{1,2,3\}$. Hence one introduces 4 divisors $E_{ij}$, $F_{ij}$, $G_{ij}$ and $H_{ij}$ on each point $P_{ij}$. They have the following properties: $E_{ij}^2=-4$, $F_{ij}^2=G_{ij}^2=H_{ij}^2=-1$, $E_{ij}F_{ij}=E_{ij}G_{ij}=E_{ij}H_{ij}=1$. The strict transform on $\widetilde{R}$ of the lines $l_i$ and $m_k$ are the followings:
$$\begin{array}{lll}
\widetilde{l_1}&:=&h-E_{12}-F_{12}-2G_{12}-H_{12}-E_{13}-F_{13}-G_{13}-2H_{13}-E_{011};\\
\widetilde{l_2}&:=&h-E_{12}-F_{12}-G_{12}-2H_{12}-E_{23}-F_{23}-2G_{23}-H_{23}-E_{\alpha 0-1};\\
\widetilde{l_3}&:=&h-E_{13}-F_{13}-2G_{13}-H_{13}-E_{23}-F_{23}-G_{23}-2H_{23}-E_{110};\\
\widetilde{m_1}&:=&h-E_{12}-2F_{12}-G_{12}-H_{12}-E_{110};\\
\widetilde{m_2}&:=&h-E_{23}-2F_{23}-G_{23}-H_{23}-E_{011};\\
\widetilde{m_3}&:=&h-E_{13}-2F_{13}-G_{13}-H_{13}-E_{\alpha0-1}.
\end{array}$$

We denote by $C$ the nodal cubic in the pencil $\mathcal{P}_{\alpha}$ which is in the branch locus of the cover $X\ra \mathbb{P}^2$, by $N$ its node and by $E_N$ the exceptional divisor over this point. The strict transform of $C$ on $\widetilde{R}$ is 
$$\begin{array}{ll}\widetilde{C}=&3h-E_{12}-2F_{12}-G_{12}-H_{12}-E_{23}-2F_{23}-G_{23}-H_{23}+\\&-E_{13}-2F_{13}-G_{13}-H_{13}--E_{011}-E_{110}-E_{\alpha0-1}-2E_N.\end{array}$$
The canonical bundle of $\widetilde{R}$ is 
$$\begin{array}{ll}K_{\widetilde{R}}=&-3h+E_{12}+2F_{12}+2G_{12}+2H_{12}+E_{23}+2F_{23}+2G_{23}+2H_{23}+\\&+E_{13}+2F_{13}+2G_{13}+2H_{13}+E_{\alpha0-1}+E_{110}+E_{011}+E_N.\end{array}$$ 

The surface $X$ admits a cover involution $\iota$, such that $X/\iota=\widetilde{R}$. It fixes the inverse image of the strict transform of the lines $l_i$, $i=1,2,3$, of the curves $E_{ij}$, $1\leq i<j\leq 3$ and of the strict transform of $\widetilde{C}$. In particular $\iota$ fixes 7 rational curves. We observe that the inverse image of the strict transform of the lines $m_i$, $i=1,2,3$ consists, for each $i$, of two disjoint rational curves switched by $\iota$. In particular $\iota^*$ does not act as the identity on $NS(X)$.

Let us fix the following notation on $X$. We denote by $\Theta_i^{(1)}$, $i\in\Z/12\Z$ the components of the fiber of type $I_{12}$, by $\Theta_j^{(h)}$, $j\in\Z/3\Z$ and $h=2,3$ the components of the fibers of type $I_3$ and by $\Theta_k^{(4)}$, $k\in \Z/2\Z$ the components of the fiber f type $I_2$. Moreover we denote by $t_0$ the zero section and we number the components of the reducible fibers in such a way that $t_0$ meets the component $\Theta_0$ of each reducible fiber. The 3-torsion sections will be called $t_1$ and $t_2$. Both $t_1$ and $t_2$ meet the zero component of the fiber of type $I_2$. The non-zero intersections with the other fibers are the following: $t_1\Theta_4^{(1)}=t_1\Theta_2^{(2)}=t_1\Theta_2^{(3)}=t_2\Theta_8^{(1)}=t_2\Theta_1^{(2)}=t_2\Theta_1^{(3)}=1$. 
Given a $2:1$ map $\pi:X\ra \widetilde{R}$, with cover involution $\iota$, we recall that for every curve $D$ in $X$, $\pi_*(D)=\delta(\pi(D))$ where $\delta\in\{1,2\}$ is the degree of the map $\pi$ restricted to $D$. So, for example, we have $\pi_*(\Theta_0^{(1)})=E_{12}$, $\pi_*(\Theta_1^{(1)})=2G_{12}$, $\pi_*(t_0)=2F_{12}$, $\pi_*(\Theta_0^{(2)})=\pi_*(\Theta_0^{(3)})=\widetilde{m_1}$.

\subsection{An elliptic fibration on $X$ induced by a conic bundle on $R$}\label{subsec: example conic bunldes} 

Let us consider on $\mathbb{P}^2$ the pencil of lines through $P_{12}$. Since $P_{12}$ is a base point of the pencil of cubics $\mathcal{P}_{\alpha}$, the pencil of lines through $P_{12}$ on $\mathbb{P}^2$ induces a fibration with rational fibers on $R$.
The class of the strict transform of the lines through $P_{12}$ on $R$ is $B:=h-E_{12}-F_{12}$. We observe that $B^2=0$ and $B\cdot F_R=2$, so  $\varphi_{|B|}:R\ra\mathbb{P}^1$ is a conic bundle on $R$. The curve $E_{12}$ is a section of the conic bundle and the fibers corresponding to the lines $l_1$, $l_2$ and $m_1$ are reducible. Indeed they are reducible fibers (of the conic bundle) of type $A_4$, $A_4$ and $A_3$ respectively (with the notation of Figure \ref{cbred}). The $(-2)$-components of the 2 reducible fibers of type $A_4$ (of the conic bundle) are contained in the fiber $I_6$ (of $\mathcal{E}_R$) and so they are branch curves of the cover $X\ra R$. The $(-2)$-component of the fiber of type $A_3$ (of the conic bundle), i.e. the strict transform of the line $m_1$, is a component of the fiber of type $I_3$ (of $\mathcal{E}_R$) and thus it is not a branch curve for the double cover $X\ra R$.

The conic bundle associated to the pencil of lines through $P_{12}$ induces an elliptic fibration on $X$, as in Proposition \ref{prop: (generalized) conic bundles and  splitting genus 1 pencil induces genus 1 fibration}. The reducible fibers of this fibration can be computed by Table \ref{table: red fibers conic bundle and elliptic} and they are $2I_2^*+I_4$. The line through $P_{12}$ and $N$ (the singular point of the branch fiber of type $I_1$) introduces an extra singular fiber: it corresponds to a smooth fiber of the conic bundle, which passes through a singular point of the fiber $(F_R)_{b_2}$ of type $I_1$ and which meets the other branch fiber, $(F_R)_{b_1}$ in two smooth points. By Table \ref{table:  extra singular fibers} this induces a fiber of type $I_2$ on the elliptic fibration on $X$. Similarly there exist two lines through $P_{12}$ which are tangent to $\widetilde{C}$ in a smooth point (different from $P_{12}$). Each of these lines gives an extra singular fibers on the elliptic fibration on $X$, which is of type $I_1$.

So the elliptic fibration induced on $X$ by the conic bundle of lines through $P_{12}$ has $2I_2^*+I_4+I_2+2I_1$ as singular fibers. Since the sum of the Euler characteristic of these fiber is 24, we know that there are no other extra singular fibers. Moreover the rank of the trivial lattice of an elliptic fibration with fibers  $2I_2^*+I_4+I_2+2I_1$ is 18 and the Picard number of $X$ is 19, so the Mordell Weil rank of this elliptic fibration is one. On the other hand it is known, \cite[Case 2031, Table 1]{Shim}, that the torsion part of the Mordell-Weil group of an elliptic fibration with these reducible fibers is $\Z/2\Z$. Hence the Mordell--Weil group of this elliptic fibration is $\Z\times \Z/2\Z$.

\subsection{An elliptic fibration on $X$ induced by a generalized conic bundle on $\widetilde{R}$}\label{subsec: examples generalized conic bundle}

Let us consider on $\mathbb{P}^2$ the pencil of lines through $N$. Since $N$ is not a base point of the pencil of cubics $\mathcal{P}_{\alpha}$, the pencil of lines through $N$ has a base point on $R$ and therefore does not induce a fibration on $R$. The class of the strict transform of the lines through $N$ on $\widetilde{R}$ is $G:=h-N$. We observe that $G^2=0$ and $G\cdot (-K_{\widetilde{R}})=2$, so $G$ is a generalized conic bundle on $\widetilde{R}$. The map $\varphi_{|G|}:\widetilde{R}\ra\mathbb{P}^1$ is a fibration whose generic fiber is a smooth rational curve. The curves $\widetilde{C}$, $\widetilde{m_1}$, $\widetilde{m_2}$ and $\widetilde{m_3}$ are sections of this fibration. The pullback of the lines connecting $N$ with one of the base points of the pencil $\mathcal{P}_{\alpha}$
are reducible fibers of the fibration $\varphi_{|G|}$. In particular each line connecting $N$ with a point $Q_{abc}$ (where $a,b,c\in\{0,1,-1,\alpha\}$) corresponds to a fiber with two components meeting in a point and each with self-intersection $-1$. Each line connecting $N$ through a point $P_{ij}$ corresponds to a reducible fiber with 5 components: one is the strict transform of the line and has self-intersection $-1$, the others are $E_{ij}$, $F_{ij}$, $G_{ij}$, $H_{ij}$. The strict transform of the line meets in one point $E_{ij}$ and no other components.

The generalized conic bundle associated to the pencil of lines through $N$ induces an elliptic fibration on $X$, as in Proposition \ref{prop: (generalized) conic bundles and  splitting genus 1 pencil induces genus 1 fibration}. The reducible fibers of this fibration  are 3 fibers of type $I_2$ (corresponding to the reducible fibers of $\varphi_{|G|}:\widetilde{R}\ra\mathbb{P}^1$ with two components, i.e. to the lines $\overline{NQ_{abc}}$) and 3 fibers of type $I_0^*$ (corresponding to the reducible fibers of $\varphi_{|G|}:\widetilde{R}\ra\mathbb{P}^1$ with five components, i.e. to the lines $\overline{NP_{ij}}$). The curves $\widetilde{m_i}$ split in the double cover, and thus each $m_i$ corresponds to two sections switched by the cover involution. Hence the elliptic fibration induced on $X$ by the pencil of lines through $N$ on $\mathbb{P}^2$ has $3I_0^*+3I_2$ as singular fibers and the rank of the Mordell--Weil group is 2.

\subsection{An elliptic fibration on $X$ of type 2) induced by a splitting genus 1 pencil}\label{subsec: example splitting genus 1 pencil}
Let us consider the class on $X$ $$D_2:=t_0+\Theta_0^{(1)}+\Theta_1^{(1)}+\Theta_2^{(1)}+\Theta_3^{(1)}+\Theta_4^{(1)}+t_1+\Theta_0^{(4)}.$$
This class has self-intersection 0, is effective and primitive. Thus it is the class of the fiber of a genus 1 fibration $\varphi_{|D_2|}:X\ra\mathbb{P}^1$. The curve $\Theta_5^{(1)}$ is a rational curve which intersects each curve in $|D_2|$ in one point, thus it is a section of the fibration, that can be chosen to be the zero section. We observe that $\iota(D_2)=D_2$, so the fibration $\varphi_{|D_2|}$ is either of type 1) or of type 2).
By the definition of $D_2$ it is clear that the fibration $\varphi_{|D_2|}$ has a fiber of type $I_8$ and the action of $\iota$ on this fiber preserves all the components. Hence $\iota$ does not act as the hyperelliptic involution on the fiber of type $I_8$ (indeed the the hyperelliptic involution acts on a fiber of type $I_8$ fixing the zero component and the fourth-component and switching three pairs of the others).
So the fibration $\varphi_{|D_2|}$ is a fibration of type 2). Let us consider now the push-down of $D_2$ on $\widetilde{R}$:
$$\begin{array}{lll}\pi_*(D_2)&=&2F_{12}+E_{12}+2G_{12}+\widetilde{l_1}+2H_{13}+E_{13}+2F_{13}+\widetilde{C}=\\& &4h-E_{12}-F_{12}-G_{12}-2H_{12}-E_{13}-F_{13}-2G_{13}-H_{13}+\\&&-E_{23}-2F_{23}-G_{23}-H_{23}-2E_{011}-E_{110}-E_{\alpha0-1}-2E_{N}.\end{array}$$
We observe that $(\pi_*{D_2})^2=0$ and $(\pi_*{D_2})(-K_{\widetilde{R}})=0$. So $\pi^*(D_2)$ is a splitting genus 1 pencil on $\widetilde{R}$.  
The class $\pi_*{D_2}$ on $\widetilde{R}$ is mapped by $\beta_1\circ\beta_2$ to the quartics in $\mathbb{P}^2$ which satisfy the following conditions:
they have a node in $Q_{011}$ and a node in $N$; they pass through $P_{12}$ with tangent direction $l_2$, through $P_{13}$ with tangent direction $l_3$, through $P_{23}$ with tangent direction $m_2$ and through $Q_{110}$ and $Q_{\alpha0-1}.$
The quartics which satisfy these conditions have genus 1 (since they are plane quartics with two nodes). The branch locus of $X\ra\widetilde{R}$ consists of the curves $\widetilde{l}_h$ and $E_{ij}$ for $h=1,2,3$ and $0\leq i<j\leq 3$ and $\widetilde{C}$. Since $\pi_*(D_2)\cdot \widetilde{l_h}=0$ for $h=1,2,3$, $\pi_*(D_2)\cdot E_{ij}=0$, $0\leq i<j\leq 3$, and $\pi_*(D_2)\cdot \widetilde{C}=0$, the curves in the linear system $|\pi_*(D_2)|$ do not intersect the branch locus of the double cover and so their inverse images consist, each, of two disjoint copies of each of them, which is consistent with the fact that $\pi_*(D_2)$ is a splitting genus 1 system of curves.

The fibration $\varphi_{|D_2|}$ has one fiber of type $I_8$ (whose components are $t_0$,$\Theta_0^{(1)}$, $\Theta_1^{(1)}$, $\Theta_2^{(1)}$, $\Theta_3^{(1)}$, $\Theta_4^{(1)}$, $t_1$, $\Theta_0^{(4)}$), one fiber of type $I_6$ (whose components are $\Theta_6^{(1)}$, $\Theta_7^{(1)}$, $\Theta_8^{(1)}$, $\Theta_9^{(1)}$, $\Theta_{10}^{(1)}$ and $D_2-(\Theta_6^{(1)}+\Theta_7^{(1)}+\Theta_8^{(1)}+\Theta_9^{(1)}+\Theta_{10}^{(1)})$), two fibers of type $I_2$ (whose components are $\Theta_1^{(2)}$, $D_2-\Theta_1^{(2)}$ and $\Theta_1^{(3)}$, $D_2-\Theta_1^{(3)}$). We observe that $\iota$ preserves the fibers of type $I_8$ and $I_6$ and permutes the fibers of type $I_2$. Thus the action of $\iota$ exhibits $X$ as double cover of a rational elliptic fibration with singular fibers $I_4+I_3+I_2+3I_1$ and the branch fibers of this double covers are the fibers of type $I_4$ and $I_3$. The rational elliptic fibration with fibers of type $I_4+I_3+I_2+3I_1$ has a Mordell--Weil group of rank 2, thus the rank of the Mordell Weil group of the elliptic fibration $\varphi_{|D_2|}:X\ra\mathbb{P}^1$ is at least two. On the other hand, the Picard number of $X$ is 19, the trivial lattice of a fibration with reducible fibers $I_8+I_6+2I_2$ has rank 16 and so the Mordell--Weil rank is 3=19-16. Hence there is at least one section of the rational elliptic fibration with fibers $I_4+I_3+I_2+3I_1$ which splits in the double cover. Indeed the curves $\Theta_0^{(2)}$ and $\Theta_0^{(3)}$ on $X$ are both sections of $\varphi_{|D_2|}:X\ra\mathbb{P}^1$ and are switched by $\iota$, hence both of them are mapped to the same section of the rational elliptic fibration with fibers $I_4+I_3+I_2+3I_1$.

\subsection{An elliptic fibration on $X$ of type 3)}\label{subsec: example type 3) fibration}

Let us consider the following class on $X$: $$D_3:=t_0+\Theta_0^{(2)}+\Theta_2^{(2)}+t_1+\Theta_0^{(4)}.$$
This class has self-intersection 0, is effective and primitive. Thus it is the class of the fiber of a genus 1 fibration $\varphi_{|D_3|}:X\ra\mathbb{P}^1$. The curve $\Theta_0^{(1)}$ is a rational curve which intersects $D_3$ in one point, thus it is a section of the fibration, that can be chosen to be the zero section. We observe that $\iota(D_3)=t_0+\Theta_0^{(3)}+\Theta_2^{(3)}+t_1+\Theta_0^{(4)}\neq D_3$, so the fibration $\varphi_{|D_3|}$ is of type 3). We observe moreover that $D_3\cdot \iota(D_3)=2$. So the generic member of the linear system $|D_3+\iota(D_3)|$ is a smooth curve of genus 3. 

Let us consider now the push down of $D_3$ on $\widetilde{R}$: it is $B_3:=\pi_*(D_3)$ and we observe that $\pi^*(B_3)=D_3+\iota(D_3)$:
$$\begin{array}{lll}B_3:=\pi_*(D_3)&=&2F_{12}+\widetilde{m_1}+\widetilde{m_3}+2F_{13}+\widetilde{C}=\\& &5h-2E_{12}-2F_{12}-2G_{12}-2H_{12}-2E_{13}-2F_{13}-2G_{13}-2H_{13}+\\&&-E_{23}-2F_{23}-G_{23}-H_{23}-2E_{110}-E_{011}-2E_{\alpha0-1}-2E_{N}.\end{array}$$

The class $B_3$ has the following properties $B_3^2=2$ and $B_3\cdot (-K_{\widetilde{R}})=2$. Since the generic member of the linear system $|B_3|$ is smooth, by adjunction it has genus 1, as expected. The curves in $|B_3|$ intersect the branch curves of the double cover $X\ra \widetilde{R}$ in the following way: $B_3\cdot \widetilde{l_i}=B_3\cdot \widetilde{C}=B_3\cdot E_{23}=0$, $B_3\cdot E_{12}=B_3\cdot E_{13}=2$. So the curve $B'$ in $X$ which is the double cover of a generic member of $|B_3|$ is a double cover of a curve of genus 1, branched in 4 points, hence it is a curve of genus 3 as expected. The sublinear system of $|B_3|$ given by the curves which meet $E_{12}$ in one point with multiplicity 2 and $E_{13}$ in one point with multiplicity 2, corresponds to the curves which split in the double cover and thus to the curves whose inverse image is the sum of a curve in $|D_3|$ and a curve in $|\iota(D_3)|$.

The class $B_3$ on $\widetilde{R}$ is mapped by $\beta_1\circ\beta_2$ to the quintics in $\mathbb{P}^2$ which satisfy the following conditions:
they have a node in $P_{12}$, $P_{13}$, $Q_{110}$, $Q_{\alpha0-1}$, $N$; they pass through $P_{23}$ with tangent direction $m_2$ and $Q_{011}$.
The quintics which satisfy these conditions have genus 1 (since they are plane quintics with 5 nodes).

The fibration $\varphi_{|D_3|}$ has one fiber of type $I_5$ (whose components are $t_0$, $\Theta_0^{(2)}$, $\Theta_2^{(2)}$, $t_1$, $\Theta_0^{(4)}$), one fiber of type $I_4$ (whose components are $\Theta_1^{(1)}$, $\Theta_2^{(1)}$, $\Theta_3^{(1)}$ and $D_3-\Theta_1^{(1)}-\Theta_2^{(1)}-\Theta_3^{(1)}$), one fiber of type $I_8$ (whose components are $\Theta_5^{(1)}$, $\Theta_6^{(1)}$, $\Theta_7^{(1)}$, $\Theta_8^{(1)}$, $\Theta_9^{(1)}$,$\Theta_{10}^{(1)}$, $\Theta_{11}^{(1)}$ and $D_3-\Theta_5^{(1)}-\Theta_6^{(1)}-\Theta_7^{(1)}-\Theta_8^{(1)}-\Theta_9^{(1)}-\Theta_{10}^{(1)}-\Theta_{11}^{(1)}$) and one fiber of type $I_2$ (whose component are $\Theta_1^{(3)}$ and $D_3-\Theta_1^{(3)}$). The classes $\Theta_0^{(1)}$, $\Theta_4^{(1)}$, $\Theta_0^{(3)}$, $\Theta_2^{(3)}$ and $t_2$ are sections of the fibration.

\end{document}